\documentclass[12pt]{amsart}
\usepackage{amssymb,amsmath,amsthm,fullpage}
\usepackage{pdfsync,color}
\usepackage{enumerate}
\usepackage{pdfsync}
\usepackage{mathrsfs}
\usepackage[bookmarksnumbered,colorlinks,plainpages,backref]{hyperref}
\usepackage[normalem]{ulem}

\numberwithin{equation}{section}

\newcommand{\C}{\mathbb C}
\DeclareMathOperator{\supp}{supp}
\DeclareMathOperator{\Imm}{Im}

\DeclareMathOperator{\Rre}{Re}
\DeclareMathOperator{\Dom}{Dom}

\newcommand{\p}{\partial}

\newcommand{\z}{\bar z}

\newcommand{\dbar}{\bar\partial}
\newcommand{\dbars}{\bar\partial^*}

\newcommand{\dbarb}{\bar\partial_b}
\newcommand{\dbarbs}{\bar\partial^*_b}

\newcommand{\Boxb}{\Box_b}
\newcommand{\vp}{\varphi}

\newcommand{\atopp}[2]{\genfrac{}{}{0pt}{2}{#1}{#2}}
\newcommand{\nn}{\nonumber}

\newcommand{\I}{\mathcal{I}}

\newcommand{\la}{\langle}
\newcommand{\ra}{\rangle}

\newcommand{\LL}{\bar L}

\DeclareMathOperator{\opL}{\mathcal{L}}
\newcommand{\om}{\omega}
\newcommand{\omb}{\bar\om}

\newtheorem{thm}{Theorem}[section]
\newtheorem{prop}[thm]{Proposition}
\newtheorem{proposition}[thm]{Proposition}

\newtheorem{lemma}[thm]{Lemma}

\newtheorem{corollary}[thm]{Corollary}

\newtheorem*{theorem*}{Theorem}

\theoremstyle{definition}
\newtheorem{defn}[thm]{Definition}
\newtheorem{definition}[thm]{Definition}

\theoremstyle{remark}

\newtheorem{remark}[thm]{Remark}

\newcommand{\Om}{\Omega}

\newcommand{\pbb}{\bar{\partial}_b}
\newcommand{\pbba}{\bar{\partial}_b^{*}}
\newcommand{\pbbpm}{\bar{\partial}_{b,{t}}^{*}}
\newcommand{\dbarbsvp}{\dbar_{b,{t}}^*}

\newcommand{\lla}[1]{\left\{ #1\right\} }

\newcommand{\bkj}{b^{\bar{k}j}}
\newcommand{\bjk}{b^{\bar{j}k}}

\newcommand{\DQ}{\Dom(\pbb)\cap \Dom(\pbba)}
\newcommand{\opH}{\mathcal{H}}
\newcommand{\opC}{\mathcal{C}}

\newcommand{\bna}{\overline{\nabla}}

\newcommand{\pare}[1]{\left( #1\right)}
\newcommand{\corc}[1]{\left[ #1\right]}
\newcommand{\ip}[1]{ \left\langle #1 \right\rangle }
\newcommand{\intprod}[2]{ \left\langle #1 , #2 \right\rangle }
\newcommand{\ba}[1]{ \overline{#1}}

\newcommand{\gqt}{G_{q,{t}}}
\newcommand{\gqtd}{G_{q,{t}}^{\delta}}
\newcommand{\gqtdn}{G_{q,{t}}^{\delta,\nu}}
\newcommand{\gqtzn}{G_{q,{t}}^{0,\nu}}
\newcommand{\vab}[1]{\left| #1 \right| }
\newcommand{\normm}[1]{ \left\| \hspace{-1pt}\vab{ #1 }\hspace{-1pt} \right\|  }

\newcommand{\norm}[1]{{\| {#1} \|}}
\newcommand{\Norm}[1]{\big\|\hspace{-1.15pt}\big| #1 \big\|\hspace{-1.2pt}\big|}

\newcommand{\vt}{\vartheta}
\newcommand{\ve}{\varepsilon}

\newcommand{\lb}{\bar{L}}
\newcommand{\lba}{\bar{L}^{*,\phi}}
\newcommand{\pb}{\dbar}
\DeclareMathOperator{\diag}{diag}

\begin{document}

\title{Closed range estimates for $\bar\partial_b$ on CR manifolds of hypersurface type}%

\author{Joel Coacalle and Andrew Raich}%
\address{Universidade Federal de S\~ao Carlos, Departamento de Matem\'atica, Rodovia Washington Luis, Km 235 - Caixa Postal 676}
\address{SCEN 327, 1 University of Arkansas, Fayetteville, AR 72701}%
\email{joelportada@dm.ufscar.br, araich@uark.edu}%

\thanks{This work was completed while the first author visited the University of Arkansas and also while the second author visited
the Universidade Federal de S\~ao Carlos. 
The authors wish to express their deep gratitude to both of these institutions.}%
\thanks{Work supported in part by CAPES (88881.135461/2016-01) and FAPESP (grant number 2018/02663-0).}

\begin{abstract}The purpose of this paper is to establish sufficient conditions for closed range estimates 
on $(0,q)$-forms, for some \emph{fixed} $q$, $1 \leq q \leq n-1$, for $\dbar_b$ in both $L^2$ and 
$L^2$-Sobolev spaces in embedded, not necessarily pseudoconvex CR manifolds of hypersurface type. The condition, named
weak $Y(q)$, is both more general than previously established sufficient conditions 
and easier to check. Applications
of our estimates include estimates for the Szeg\"o projection as well as an argument that the harmonic forms have the same regularity as the
complex Green operator. We use a microlocal argument and carefully construct a norm that is well-suited for a microlocal decomposition of
form. We do not require that the CR manifold is the boundary of a domain. Finally, we provide an example that demonstrates that
weak $Y(q)$ is an easier condition to verify than earlier, less general conditions.
\end{abstract}

\maketitle

%
%
\section{Introduction}

In this paper, we {show that the tangential Cauchy-Riemann operator has closed range on $(0,q)$-forms, for a \emph{fixed} $q$,
$1 \leq q \leq n-1$, in $L^2$ and $L^2$-Sobolev spaces
on a general class of embedded CR manifolds of hypersurface type that}
satisfy a general geometric condition called \emph{weak $Y(q)$}.  We work on a 
smooth CR submanifold $M \subset \C^n$ that may be neither pseudoconvex nor the boundary of a domain. The weak $Y(q)$ condition,
first written down by Harrington and Raich \cite{HaRa15} 
{and applied to} boundaries of domains in Stein manifolds, is the most general known condition that
{ensures} closed range of the tangential Cauchy-Riemann operator on $(0,q)$-forms. We also provide an example that shows that the generality
provided by the definition makes it easier to verify than {previous and more restrictive} conditions.
Additionally, we show that for any Sobolev level, there is a weight such that the {(weighted)} complex
Green operator {(inverse to the weighted Kohn Laplacian)} is continuous and the harmonic forms in this weighted space are elements of the 
prescribed Sobolev space.

This paper generalizes both \cite{HaRa11} and \cite{HaRa15} in the following ways. We do not require our CR manifold to be the boundary of
a domain. In effect, we translate the $\dbar$-techniques of \cite{HaRa15} to the microlocal setting. In \cite{HaRa11}, they
prove results akin to our main results, but the ``weak $Y(q)$" condition they define is more restrictive than the
weak $Y(q)$ condition here. Additionally, we use a reengineered elliptic regularization argument to show that (weighted) harmonic $(0,q)$-forms
are smooth, a fact not mentioned in \cite{HaRa11,HaRa15}. Additionally, we are careful to monitor the regularized operators and the fact
that they preserve orthogonality with the space of (weighted) harmonic forms, a fact that has not been observed before (in part because we
prove smoothness of harmonic forms early in regularization process).

Throughout this paper, we will consider $M \subset\C^N$ 
being a $2n-1$ real dimension, $C^\infty$, compact, orientable CR-manifold, $N \geq n$ of hypersurface type. This last condition means
that the CR dimension of $M$ is $n-1$ so that the complex tangent bundle splits into a complex subbundle of dimension $n-1$,
the conjugate subbundle, and one totally real direction.  An appropriate restriction of 
the $\dbar$-complex to $M$ yields the $\dbarb$-complex.

The $\dbarb$-operator was introduced by Kohn and Rossi \cite{KoRo65} to study the boundary values of holomorphic functions on domains in 
$\C^n$, and it was soon realized that the $\dbarb$-complex was deeply intertwined with
the geometry and potential theory of such domains and their boundaries. The story of the $L^2$-theory of the $\dbarb$-operator
begins with Shaw \cite{Sha85} and Boas and Shaw \cite{BoSh86} (in the top degree) on boundaries of pseudoconvex domains in $\C^n$ and
with Kohn \cite{Koh86} on the boundaries of pseudoconvex domains in Stein manifolds. Nicoara \cite{Nic06} established closed range
for $\dbarb$ (at all form levels) on smooth, embedded, compact, orientable CR manifolds of hypersurface dimension in the case that $n\geq {3}$ and
Baracco \cite{Bar12} established the $n={2}$ case. Thus, from the point of view closed range, the pseudoconvex case is completely understood.

Harrington and Raich \cite{HaRa11} began an investigation of the $\dbarb$-problem on non-pseudoconvex CR manifolds of hypersurface type. Specifically,
they fixed a level $q$, $1 \leq q \leq n-2$, and sought a general condition that sufficed to prove closed range of $\dbarb$ on $(0,q)$-forms
(and in $L^2$-Sobolev spaces in suitably weighted spaces). They worked on CR manifolds of hypersurface type, 
{and our results generalize theirs by showing that the conclusions they draw are still true with a weaker hypothesis, namely, 
the weak $Y(q)$ condition from \cite{HaRa15}}. The analysis in \cite{HaRa15} is loosely based
on the ideas of Shaw and do{es} not use a microlocal argument, but {rather} $\dbar$-methods. This requires the CR manifold to be the boundary of a domain,
a {hypothesis that we relax}.
The name weak $Y(q)$ stems from the fact that {it} is a weakening of the classical 
$Y(q)$ condition, a geometric condition that is equivalent to the complex Green operator satisfying $1/2$-estimates on $(0,q)$-forms. The \emph{complex Green
operator}, when it exists, is the name for the (relative) inverse to $\Boxb$ in $L^2_{0,q}(M)$ and denoted by $G_q$. 

Our methods involve a microlocal argument in the spirit of \cite{Nic06,Rai10,HaRa11} and a recently reengineered elliptic regularization that not only
allows for a weighted complex Green operator to solve the $\dbarb$-problem in a given $L^2$-Sobolev space, but also shows that the 
weighted $L^2$-harmonic forms reside in that Sobolev space \cite{KhRa20,HaRa20SCRE}. 
This last fact is not clear from the elliptic regularization methods used in
\cite{Nic06,HaRa11}. For a discussion of the weak $Y(q)$ condition and its related, non-symmetrized version, weak $Z(q)$, please see
\cite{HaRa11,HaRa15,HaPeRa15, HaRa18,HaRa19} and for discussion on the elliptic regularization method, \cite{HaRa20SCRE,KhRa20}.

The outline of the argument is as follows: we start by proving a basic identity that is well suited to the geometry of $M$. 
The problem
with basic identities for $\dbarb$ is that the Levi form appears with in a term that also contains the derivative in the totally real direction. 
The 
microlocal argument is used to control this term -- specifically, we construct a norm based on a microlocal decomposition of our form which
allows us to use a version of the sharp G{\aa}rding's inequality and eliminate the $T$ from the inner product term. This allows us to prove a
basic estimate (Proposition \ref{prop:mainestimate})
from the basic identity and the main results are due to careful applications of the basic estimate.

The outline of the paper is the following. We conclude this section with statements of our main theorems.  
In Section \ref{sec:defs}, we define our notation. In Section \ref{sec:pseudo}, we give some computations in local coordinates
and the microlocal decomposition. 
In Section \ref{sec: basic estimate}, we prove the basic
estimate, Proposition \ref{prop:mainestimate}. 
In Section \ref{sec:main theorem, weighted}, we prove the Theorem \ref{thm:mainthm, Sobolev}. Many of the consequences
of Theorem \ref{thm:mainthm, Sobolev} use identical proofs to \cite[Theorem 1.2]{HaRa11}, once
{we have completed the elliptic regularization argument, established}
the continuity of $G_{q,{t}}$ on $H^s_{0,q}(M)$, and {proved the regularity} of the weighted harmonic forms.
In Section \ref{sec:proof of main theorem}, we outline how to pass from Theorem \ref{thm:mainthm, Sobolev} to Theorem \ref{thm:mainthm, unweighted}.
We conclude the paper in Section \ref{sec:examples} with an example.

\begin{thm}\label{thm:mainthm, unweighted}
	Let $ M^{2n-1}$ be an embedded $C^\infty$, compact, orientable CR-manifold of hypersurface type that satisfies
	weak $Y(q)$ for some fixed $q$, $ 1\leq q\leq n-2 $. 
	Then the following hold:
	\begin{enumerate}
		\item The operators $\pbb:L^2_{0,q}(M)\rightarrow L^2_{0,q+1}(M) $ and $ \pbb: L^2_{0,q-1}(M)\rightarrow L^2_{0,q}(M)$ have closed range;
		\item The operators $ \pbba:L^2_{0,q+1}(M)\rightarrow L^2_{0,q}(M) $ and $ \pbba: L^2_{0,q}(M)\rightarrow L^2_{0,q-1}(M)$ have closed range;
		\item The Kohn Laplacian $ \Box_b:=\pbb\pbba+\pbba\pbb $ has closed range on $ L^2_{0,q}(M) $;
		\item The complex Green operator $G_q $ exists and is continuous on $L^2_{0,q}(M) $;
		\item The canonical solution operators, $\pbba G_{q}:L^2_{0,q}(M)\rightarrow  L^2_{0,q-1}(M)$ 
		and $ G_{q}\pbba:L^2_{0,q+1}(M)\rightarrow  L^2_{0,q}(M)$ are continuous;
		\item The canonical solution operators, $ \pbb G_{q}:L^2_{0,q}(M)\rightarrow  L^2_{0,q+1}(M)$ 
		$  G_{q}\pbb:L^2_{0,q-1}(M)\rightarrow  L^2_{0,q}(M)$ are continuous;
		\item The space of the harmonic forms $\opH_{0,q}(M) $, defined to be the (0,q)-forms annihilated by $ \pbb $ and $ \pbba $, is finite dimensional;
		\item If $\tilde q = q$ or $q+1$ and $\alpha \in L^2_{0,\tilde q}$, then there exists $u \in L^2_{0,\tilde q-1}$ so that 
		\[
		\dbarb u = \alpha
		\]
		and $\|u\|_{{0}} \leq C \|\alpha\|_{{0}}$ for some constant $C$ independent of $\alpha$;
		\item The Szeg\"o projections $S_q=I-\pbba\pbb G_q$  and $S_{q-1} = I - \dbarbs G_q \dbarb$ are continuous on $L^2_{0,q}(M)$.
	\end{enumerate}
\end{thm}
In fact, Theorem \ref{thm:mainthm, unweighted} follows immediately from Theorem \ref{thm:mainthm, Sobolev} using standard techniques
and the fact that the constructed norm $\normm{\cdot}_{{t}}$ is equivalent to the unweighted norm $\|\cdot\|_{{0}}$. 
{We denote the $L^2$ space with respect to $\normm{\cdot}_t$ by $L^2(M,\normm{\cdot}_t)$. Additionally, we
use the (equivalent) norm $\normm{\Lambda^s \cdot}_t$ on $H^s(M)$ because with it, we can obtain better constants and denote
the $H^s(M)$ with respect to this measurement by $H^s(M,\normm{\cdot}_t)$ .}
\begin{thm}\label{thm:mainthm, Sobolev}
Let $ M^{2n-1} $ be a $ C^\infty$  compact, orientable, weakly $ Y(q) $ CR-manifold of hypersurface type embedded in 
$\C^N$, $N \geq n$, and $ 1\leq q\leq n-2 $. For each $s\geq 0 $ there exists $ T_s\geq 0 $ so that the following hold:
\begin{enumerate}[i.]
	\item The operators {$\pbb:L^2_{0,q}(M,\normm{\cdot}_t)\rightarrow L^2_{0,q+1}(M,\normm{\cdot}_t)$ and 
	$ \pbb: L^2_{0,q-1}(M,\normm{\cdot}_t)\rightarrow L^2_{0,q}(M,\normm{\cdot}_t)$ have closed range}. 
	Additionally, for any $s>0$ if $ t\geq T_s $, then 
	$ \pbb:H^s_{0,q}{(M,\normm{\cdot}_t)}\rightarrow H^s_{0,q+1}{(M,\normm{\cdot}_t)} $ 
	and $ \pbb: H^s_{0,q-1}{(M,\normm{\cdot}_t)}\rightarrow H^s_{q}{(M,\normm{\cdot}_t)}$ have closed range.
	\item The operators $ \dbars_{b,{t}}:L^2_{0,q+1}{(M,\normm{\cdot}_t)}\rightarrow L^2_{0,q}{(M,\normm{\cdot}_t)} $ 
	and $ \dbars_{b,{t}}: L^2_{0,q}{(M,\normm{\cdot}_t)}\rightarrow L^2_{0,q-1}{(M,\normm{\cdot}_t)}$ 
have closed range; 
	Additionally, if $ t\geq T_s $, then $ {\pbbpm}:H^s_{0,q+1}{(M,\normm{\cdot}_t)}\rightarrow H^s_{{0,}q}{(M,\normm{\cdot}_t)} $ 
	and $ {\pbbpm}: H^s_{{0,}q}{(M,\normm{\cdot}_t)}\rightarrow H^s_{0,q-1}{(M,\normm{\cdot}_t)}$ have closed range.
	\item The Kohn Laplacian $ \Box_{b,t}:=\pbb\dbars_{b,{t}}+\dbars_{b,{t}}\pbb $ has closed range on $L^2_{0,q}{(M,\normm{\cdot}_t)} $, 
	and if $ t\geq T_s $, $\Box_{b,t}$ also has closed range on $H^s_{0,q}{(M,\normm{\cdot}_t)} $.
	\item The space of (weighted) harmonic forms $ \opH_{{t}}^q(M)$, defined to be the $(0,q)$-forms annihilated by 
	$ \pbb $ and $ \dbars_{b,{t}} $, is finite dimensional.
	\item The complex Green operator $ G_{q,{t}} $ exists and is continuous on $ L^2_{0,q}{(M,\normm{\cdot}_t)} $
	and also on $ H^s_{0,q}{(M,\normm{\cdot}_t)} $ if $ t\geq T_s $.
	\item The canonical solution operators for $ \pbb $, 
	$ \dbars_{b,{t}} G_{q,{t}}:L^2_{0,q}{(M,\normm{\cdot}_t)}\rightarrow  L^2_{0,q-1}{(M,\normm{\cdot}_t)}$ 
	and $ G_{q,{t}}\dbars_{b,{t}}:L^2_{0,q+1}{(M,\normm{\cdot}_t)}\rightarrow  L^2_{0,q}{(M,\normm{\cdot}_t)}$ are continuous.
	Additionally, $ \dbars_{b,{t}} G_{q,{t}}:H^s_{0,q}{(M,\normm{\cdot}_t)}\rightarrow  H^s_{0,q-1}{(M,\normm{\cdot}_t)}$ and $  G_{q,{t}}\dbars_{b,{t}}:H^s_{0,q+1}{(M,\normm{\cdot}_t)}\rightarrow  H^s_{0,q}{(M,\normm{\cdot}_t)}$
	are continuous if $ t\geq T_s $.
	\item The canonical solution operators for $ \dbars_{b,{t}} $, $ \pbb G_{q,{t}}:L^2_{0,q}{(M,\normm{\cdot}_t)}\rightarrow  L^2_{0,q+1}{(M,\normm{\cdot}_t)}$ 
	and  $  G_{q,{t}}\dbarb:L^2_{0,q-1}{(M,\normm{\cdot}_t)}\rightarrow  L^2_{0,q}{(M,\normm{\cdot}_t)}$ are continuous. 
	Additionally, $ \pbb G_{q,{t}}:H^s_{0,q}{(M,\normm{\cdot}_t)}\rightarrow  H^s_{0,q+1}{(M,\normm{\cdot}_t)}$ 
	and $  G_{q,{t}}\dbarb:H^s_{0,q-1}{(M,\normm{\cdot}_t)}\rightarrow  H^s_{0,q}{(M,\normm{\cdot}_t)}$
	are continuous if $ t\geq T_s $.
	\item The Szeg\"o projections $ S_{q,t}=I-\dbars_{b,{t}}\pbb G_{q,{t}} $ and $ S_{q-1,t}=I-\dbars_{b,{t}} G_{q,{t}}\pbb $ are
	 continuous on $L^2_{0,q}{(M,\normm{\cdot}_t)} $ and $L^2_{0,q-1}{(M,\normm{\cdot}_t)}$, respectively.	
	  Additionally, if $ t\geq T_s $ then $S_{q,t} $ and $S_{q-1,t}$ are continuous on $ H^s_{0,q}{(M,\normm{\cdot}_t)}$ 
	  and $H^s_{0,q-1}{(M,\normm{\cdot}_t)}$, respectively.
\end{enumerate}
\end{thm}


%
%
\section{Definitions and Notation}\label{sec:defs}

\subsection{CR manifolds}
%
%
\begin{definition}\label{def:CRstructure}
	Let $M$ a smooth manifold of real dimensional $2n-1$. M is called a \emph{CR-manifold of hypersurface type} 
	if $M$ is equipped with a subbundle of the complexified tangent bundle $\C T(M)$ denoted by $\mathbb{L}$ satisfying:
	
	\begin{enumerate}[(i)]
		\item $\dim_{\C}\mathbb{L}_x=n-1$ where $\mathbb{L}_x$ is the fiber {over} $x\in M$.
		\item $\mathbb{L}_x\cap \overline{\mathbb{L}}_x=\lla{0}$ where $\overline{\mathbb{L}}_x$ is the complex conjugate of $\mathbb{L}_x$.
		\item If $L,L'\in \mathbb{L}$ then $\corc{L,L'}:=LL'-L'L$ is in $\mathbb{L}$. 
	\end{enumerate}
	
\end{definition}
$\mathbb{L}$ is called the CR structure of $M$. 
Since $M$ is embedded in $\C^N$, 
we define $T_z^{1,0}(M)=T_z^{1,0}(\C^N)\cap T_z(M)\otimes \C $ (under the natural inclusion). Since the complex dimension of  the CR structure
is $ n-1 $ for all $ z\in M $, we can set $\mathbb{L}= T^{1,0}(M)=\bigcup_{z\in M}T^{1,0}_z(M) $, and this defines a CR structure on $ M $ 
that called the \emph{induced CR structure} on $ M $.

For this paper, we consider only smooth, orientable $CR$ manifolds of hypersurface type embedded in a complex space $\C^N${, though
our techniques should generalize to Stein manifolds, a topic that we do not pursue here to notational simplicity and clarity}. 
Let $T^{p,q}(M)$  denote the space of exterior algebra generated by $T^{1,0}(M)$ and $T^{0,1}(M)$. 
Let $\Lambda^{p,q}(M)$ denote the bundle of $(p,q)$-forms on $T^{p,q}(M)$, this is $\Lambda^{p,q}(M)$ 
consist of skew-symmetric multilinear maps of $T^{p,q}(M)$ into $\C$. Because we are in $\C^N$, our calculations do not depend on $p$, and
we therefore set $p=0$ for the remainder of the manuscript.

\subsection{{$\dbarb$ on embedded manifolds}}
Since
$M \subset\C^N$ for some $N \geq n$, and our CR structure is the induced one,  it is natural to use the induced metric on $\C T(M)$, 
denoted by $\la \cdot, \cdot \ra_x$ for each $x\in M$. The metric $\la\cdot,\cdot\ra_x$ is compatible with the
induced CR structure in the sense that the vector spaces $T^{1,0}_x$ and $T^{0,1}_x$ are orthogonal.
We use the inner product on $\Lambda^{0,q}(M)$ given by
\[
( \vp,\psi )_0 = \int_M \la \vp,\psi \ra_x \, dV
\]
where $dV$ is the volume element on $M$. The involution condition (iii) in Definition \ref{def:CRstructure} means that $ \pbb $ 
can be defined as the restriction of the Rham exterior derivative $ d $ to $ \Lambda^{0,q}(M) $.

The Hermitian inner product above gives rise to an $ L^2 $-norm $ \|\cdot\|_{{0}} $, and we also denote the closure of $ \pbb $ in this 
norm by $ \pbb $ (by an abuse of notation). In this way, $ \pbb:L^2_{0,q}(M)\rightarrow L^2_{0,q+1}(M) $ is a 
well-defined, closed, densely defined operator, and we define 
$ \pbba: L^2_{0,q+1}(M)\rightarrow L^2_{0,q}(M) $ to be {its} $ L^2 $ adjoint. 
The Kohn Laplacian $ \Box_b:L^2_{0,q}(M)\rightarrow L^2_{0,q}(M) $ is defined as
\[
\Box_b:=\pbba\pbb+\pbb\pbba.
\]

\subsection{The Levi form}
From the CR structure on $M$, there is a local orthonormal basis $L_1,...,L_{n-1}$ of the $(1,0)$-vector fields in a neighborhood
$U$ of a point $x\in M$. Let $\om_1,\dots,\om_{n-1}$ be the dual basis of $(1,0)$-forms so that $\la \om_j,L_k\ra = \delta_{jk}$. 
This means $\bar L_1,\dots,\bar L_{n-1}$ is a orthonormal basis of $T^{0,1}(U)$ with dual basis $\omb_1,\dots,\omb_{n-1}$ in $U$. Finally,
there is vector $T$, taken purely imaginary, so that $\{L_1,\dots,L_{n-1},\bar L_1,\dots,\bar L_{n-1}, T\}$ is an orthonormal basis
of $T(U)$. Since $M$ is oriented, there exists a globally defined $1$-form $\gamma$ that annihilates $T^{1,0}(M)\oplus T^{0,1}(M)$ and is normalized
so that $\la \gamma, T \ra = -1$.

\begin{defn}\label{defn: Levi form}
The \emph{Levi form} at a point $x\in M$ is the Hermitian form given by $\intprod{d\gamma_x}{L\wedge \bar{L}'}$ 
for any $L,L'\in T^{1,0}_x(U)$, and $ U $ is a neighborhood of $x\in M$.
\end{defn}
Cartan's formula implies that for any $L,L'\in T^{1,0}(M)$, we have
\begin{equation}\label{eq:cartan}
\intprod{d\gamma}{L\wedge \bar{L}'}=-\intprod{\gamma}{\corc{L,\bar{L}'}}.
\end{equation}
In local coordinates, for any $1 \leq j,k \leq n-1$, 
\[
\corc{L_j,\ba{L}_k}=c_{jk}T\ mod\ T^{1,0}(U)\oplus T^{0,1}(U)
\]
so that $\intprod{d\gamma}{L_j\wedge\ba{L}_k}=c_{jk}$. We will call $\corc{c_{jk}}_{1\leq j,k\leq n-1}$ the \emph{Levi matrix} with
respect to $L_1,...,L_{n-1},T$. 

Let $\mu_1,...,\mu_{n-1}$ be the eigenvalues of $\corc{c_{jk}}$ such that $\mu_1\leq\mu_2\leq...\leq\mu_{n-1}$. The CR structure is called (strictly) pseudoconvex in some point $ p\in M $ if the matrix $ \corc{c_{jk}(p)} $,
is positive (definite) semidefinite. If the CR structure is (strictly) pseudoconvex in every point, then it is called (strictly) pseudoconvex. 

Now, we introduce the main geometric condition for our CR manifolds, given by Harrington and Raich in \cite{HaRa15}.


\begin{definition}\label{defn:weak Z(q)}
For $1\leq q\leq n-1$ we say $M$ satisfies \emph{$Z(q)$-weakly} if there exists a real $\Upsilon\in T^{1,1}(M)$ satisfying
\begin{enumerate}[(A)]
	\item $\vab{\theta}^2\geq (i\theta\wedge\ba{\theta})(\Upsilon)\geq 0$ for all $\theta\in \Lambda^{1,0}(M)$
	\item $\mu_1+\mu_2+{\cdots}+\mu_{q}-i\ip{d\gamma_x,\Upsilon}\geq 0$ where $\mu_1,...,\mu_{n-1}$ are the eigenvalues of the Levi form at $x$
	in increasing order.
	\item $\omega(\Upsilon)\neq q$ where $\omega$ is the $(1,1)$-form associated to the induced metric on $\C T(M)$.
\end{enumerate}
We say that $M$ satisfies \emph{weak $Y(q)$} if $M$ satisfies both $Z(q)$-weakly and $Z(n-q-1)$-weakly.
\end{definition}

For example, it is easy to see that if  $ M $ is pseudoconvex, then $M$ satisfies weak $Z(q)$ for any $ 1\leq q \leq n-1$ {with $\Upsilon=0$}. 
Please see \cite{HaRa15,HaPeRa15,HaRa18} for a discussion of the weak $Z(q)$ property. 
The symmetric hypotheses on form levels
on $q$ and $n-1-q$ are necessary due a Hodge-* operator \cite{RaSt08,BiSt17}. 

\begin{remark}
If $ M $ is a CR manifold satisfying  $Y(q)$ weakly, then $ \Upsilon $ corresponding to weak $ Z(q) $, which we denote by
$\Upsilon_q$, may be unrelated to the
$ \Upsilon $ that corresponds to weak $ Z(n-q-1) $ (similarly denoted by $\Upsilon_{n-1-q}$). 
\end{remark}
Given a function $\vp$ defined near $M$, we define the two form
\[
\Theta^\vp = \frac 12 \Big(\p_b\dbarb\vp - \dbarb\p_b\vp\Big) + \frac 12 \nu(\vp)\, d\gamma
\]
where $\nu$ is the real part of the complex normal to $M$. When we work locally, we often associate $\Theta^\vp$ with 
the matrix $\Theta^\vp_{jk} = \la \Theta^\vp, L_j \wedge \bar L_k \ra$. We know that for such $\vp$
\[
\Big\la \frac 12 \big(\p\dbar\vp - \dbar\p\vp\big), L \wedge \bar L \Big\ra 
= \big\la \Theta^\vp, L \wedge \bar L\big\ra
\]
which means $\Theta^{|z|^2}= \p\dbar|z|^2=\omega$ \cite[Proposition 3.1]{HaRa11}.

%
%
\section{Local Coordinates and Pseudodifferential  Operators}\label{sec:pseudo}

\subsection{Pseudodifferential Operators}\label{subsec:PseuO}
We follow the setup {from} \cite{Rai10}. By the compactness of $M$, there exists a finite cover $\lla{U_{{\mu}}}_{{\mu}}$, so each $U_{{\mu}}$ has a special boundary system and can be parameterized by a hypersurface in $\C^n$ ($U_{{\mu}}$ may be shrunk as necessary).\\

Let $\xi=(\xi_1,...,\xi_{2n-2},\xi_{2n-1})=(\xi',\xi_{2n-1})$ be the coordinates in Fourier space so that 
$\xi'$ is the dual variable to the variables in the maximal complex tangent space and 
$\xi_{2n-1}$ is dual to the totally real part of $T(M)$, i.e., the ``bad" direction $T$. Define
\begin{eqnarray*}
\opC^+ &=&\lla{ \xi:\xi_{2n-1}\geq \dfrac{1}{2}\vab{\xi'} \text{ and } \vab{\xi}\geq 1};\ \ \ \ \opC^- =\lla{ \xi:-\xi \in \opC^+};\\
\opC^0 &=&\lla{ \xi: -\dfrac{3}{4}\vab{\xi'}\leq \xi_{2n-1}\geq \dfrac{3}{4}\vab{\xi'}}\cup\lla{\xi:\vab{\xi}\leq 1}.
\end{eqnarray*}
$ \opC^+ $ and $ \opC^- $ are disjoint, but both intersect $ \opC^0 $ nontrivially. 
Next, let $ \psi^+, \psi^- $ and $ \psi^0 $ be smooth functions on the unit sphere so that
\begin{align*}
	 \psi^+(\xi)&=1 \text{ when } \xi_{2n-1}\geq \frac{3}{4}\vab{\xi'} \text{ and } \supp\psi^+\subset\lla{ \xi: \xi_{2n-1}\geq \frac{1}{2}\vab{\xi'} };\\
	 \psi^-(\xi)&=\psi^+(-\xi);\\
	 \psi^0(\xi)& \text{ satisfies } \psi^0(\xi)^2=1-\psi^+(\xi)^2-\psi^-(\xi)^2.
\end{align*}
Extend $ \psi^+,\psi^-, $ and  $ \psi^0 $ homogeneously outside of the unit ball, i.e., if $ \vab{\xi}\geq 1 $, then
\[
\psi^+(\xi)=\psi^+(\xi/\vab{\xi}),\ \  \psi^-(\xi)=\psi^-(\xi/\vab{\xi}),\  \text{ and }\  \psi^0(\xi)=\psi^0(\xi/\vab{\xi}).
\]
Finally, extend $ \psi^+,\psi^- $ and $ \psi^0 $ smoothly inside the unit ball so that $ (\psi^+)^2+(\psi^-)^2+(\psi^0)^2=1 $ {and $\psi^+$ and $\psi^-$ are supported
away from $B(0,\frac12)$}. 
For a fixed constant $A>0$ to be chosen later, define for any $t>0$,
\begin{equation*}
\psi^+_t(\xi)=\psi^+(\xi/(tA)), \psi^-_t(\xi)=\psi^-(\xi/(tA)), \text{ and } \psi^0(\xi)=\psi^0(\xi/(tA)).
\end{equation*}
Let $ \Psi^+_t,\Psi^-_t, $ and  $ \Psi^0_t $ be the pseudodifferential operators of order zero with 
symbols $ \psi^+_t,\psi^-_t, $ and  $ \psi^0_t $, respectively. The equality $ (\psi^+_t)^2+(\psi^-_t)^2+(\psi^0_t)^2=1 $ implies that
\begin{equation*}
 (\Psi^+_t)^*\Psi^+_t+(\Psi^-_t)^*\Psi^-_t+(\Psi^0_t)^*\Psi^0_t=I.
\end{equation*}
Suppose $\psi$ and $\tilde\psi$ are cut-off functions so that $\tilde{\psi}|_{{\supp}\psi}\equiv 1$. 
If $\Psi$ and $\tilde{\Psi}$ are pseudodifferential operators with symbols $\psi$ and $\tilde{\psi}$, respectively, 
then we say that $\tilde{\Psi}$ \emph{dominates} $\Psi$.

For each ${{\mu}} $, let $ \Psi^+_{{{\mu}},t},\Psi^-_{{{\mu}},t} $, and  $ \Psi^0_{{{\mu}},t} $ be the operators $\Psi^+_t,  \Psi^-_t$, and $\Psi^0_t$, respectively, defined on 
$ U_{{\mu}} $, where $ \opC^+_{{{\mu}}},\opC^-_{{{\mu}}} $ are  
$ \opC^0_{{{\mu}}}$ be the corresponding regions of $ \xi$-space dual to $ U_{{\mu}} $.
It follows that
\begin{equation*}
(\Psi^+_{{\mu},t})^*\Psi^+_{{\mu},t}+(\Psi^-_{{\mu},t})^*\Psi^-_{{\mu},t}+(\Psi^0_{{\mu},t})^*\Psi^0_{{\mu},t}=I.
\end{equation*}

Additionally, 
let $ \tilde{\Psi}^+_{\mu,t} $ and $ \tilde{\Psi}^-_{\mu,t} $ be pseudodifferential operators that dominate $ \Psi^+_{\mu,t} $ and $ \Psi^-_{\mu,t} $ respectively (where $ \Psi^+_{\mu,t} $ and $ \Psi^-_{\mu,t} $ are defined on some $U_\mu$ ). If $\tilde{\opC}^+_\mu$ and  $\tilde{\opC}^-_\mu$ are the supports of the symbols of $ \tilde{\Psi}^+_{\mu,t} $ and $ \tilde{\Psi}^-_{\mu,t} $, respectively, then we can choose $ \lla{U_\mu} $, $ \tilde{\psi}^+_{\mu,t} $, and $ \tilde{\psi}^-_{\mu,t} $ so that the following result holds \cite{Nic06}.

\begin{lemma}[Lemma 4.3, \cite{Nic06}]
Let M be a compact, orientable, embedded CR-manifold. There is a finite open covering $ \lla{U_\mu}_\mu $ of $M$ so that if 
$ U_\mu,U_{{\mu'}} \in \lla{U_\mu}$ have nonempty intersection, then there exits a diffeomorphism $ \vartheta $ between 
$ U_\mu $ and $ U_{{\mu'}} $ with Jacobian $ \mathcal{J}_\vartheta $ such that

\begin{enumerate}[(i)]
	\item $ ^t\mathcal{J}_\vt(\opC_\mu^+)\cap \opC_{{\mu'}}^-=\emptyset $ and $ \opC_{{\mu'}}^+\cap  {}^t\mathcal{J}_\vt(\opC_{{\mu}}^-)=\emptyset $ where $ {^t\mathcal{J}_\theta} $ is the inverse of the transpose of the Jacobian of $ \vt $;\\
	\item let $ {}^\vt\Psi_{t,\mu}^+,{}^\vt\Psi_{t,\mu}^- $ and $ {}^\vt\Psi_{t,\mu}^0 $ be the transfer of $ \Psi_{t,\mu}^+,\Psi_{t,\mu}^- $ and $ \Psi_{t,\mu}^0 $, respectively via $ \vt $, then on $ \lla{ \xi:\xi_{2n-1}\geq \frac{4}{5}\vab{\xi'} \text{ and } \vab{\xi}\geq(1+\ve)tA } $, the principal symbol of $ {}^\vt\Psi_{t,\mu}^+ $is identically equal to 1, on $ \lla{ \xi:\xi_{2n-1}\leq -\frac{4}{5}\vab{\xi'} \text{ and } \vab{\xi}\geq(1+\ve)tA } $, the principal symbol of $ {}^\vt\Psi_{t,\mu}^- $is identically equal to 1, and on $ \lla{ \xi:-\frac{1}{3}\vab{\xi'}\leq \xi_{2n-1}\leq \frac{1}{3}\vab{\xi'} \text{ and } \vab{\xi}\geq(1+\ve)tA } $, the principal symbol of $ {}^\vt\Psi_{t,\mu}^0 $is identically equal to 1, where $\ve>0$ and can be very small.\\
	\item Let $ {}^\vt\tilde{\Psi}_{t,\mu}^+,{}^\vt\tilde{\Psi}_{t,\mu}^- $ be the transfer via $ \vt $ of $ \tilde{\Psi}_{t,\mu}^+,\tilde{\Psi}_{t,\mu}^- $ respectively.	Then the principal symbol of $ {}^\vt\tilde{\Psi}_{t,\mu}^+ $ is identically 1 on $ \opC_{{\mu'}}^+ $ 
and the principal symbol of ${}^\vt\tilde{\Psi}_{t,\mu}^-$ is identically 1 on $\opC_{{\mu'}}^-$;\\
	\item $\tilde{\opC}_{{\mu'}}^+\cap \tilde{\opC}_{{\mu'}}^-=\emptyset$.
\end{enumerate}

\end{lemma}

We will suppress the left superscript $ \vt $ as it should be clear from the context which pseudodifferential operator must be transferred. 
If $P$ is any of the operators $ \Psi_{t,\mu}^+,\Psi_{t,\mu}^- $ or $ \Psi_{t,\mu}^0 $ then it is immediate that
\begin{equation*}
D_{\xi}^\alpha\sigma(P)=\dfrac{1}{\vab{t}^\alpha} q_\alpha(x,\xi)
\end{equation*}
for $ \vab{\alpha}\geq 0 $, where $ q_\alpha(x,\xi) $ is bounded independently of $t$.

\subsection{Norms}
%
%
If $ \phi  $ is a real function defined on $ M $, then define the weighted Hermitian inner for 
$ (0,q) $-forms $ f $ and $ g $, denoted by $ \pare{f,g}_\phi$ by $ \pare{f,g}_\phi =\pare{e^{-\phi}f,g}_0 $.
For example, if $ f=\sum_{J\in \I_q}f_J\bar{\omega}^J $ is a (0,q)-form supported on neighborhood $ U $, where 
$ \I_q = \{J = (j_1,\dots,j_q) : 1 \leq j_1 < j_2 < \cdots < j_q\}$  and $ \omega^J=\omega_{j_1}\wedge\dots\wedge\omega_{j_q} $.
The weighted $L^2$-norm on $(0,q)$-forms is 
$ \norm{f}_\phi^2:=\sum_{J\in \I_q}\norm{f_J}_\phi^2$ where $ \norm{f_J}_\phi^2=\int_M \vab{f_J}^2e^{-\phi}dV $,
and we denote the corresponding weighted $L^2$ space by $L^2_{0,q}(M,e^{-\phi})$.

We now construct a norm that is well adapted to the microlocal analysis. 
Let   {
$\{U_\mu\}_\mu$ be an covering of $M$ that admits the family of pseudodifferential operators $\{\Psi_{\mu,t}^{+},\ \Psi_{\mu,t}^{-} ,\  \Psi_{\mu,t}^{0}\}$
and a partition of unity $ \lla{\zeta_\mu}_\mu $ subordinate to the cover satisfying $ \sum_{\mu}\zeta_\mu^2=1 $.
F}or each $ \mu $ let $ \tilde{\zeta}_\mu $ be a cutoff function that dominates $ \zeta_\mu $ such that $\supp\tilde{\zeta}_\mu\subset U_\mu $, and $ \phi^+ $, $ \phi^- $ smooth functions defined on M. We define the global inner product and norm as follows:
\begin{align*}
\pare{f,g}_{\phi^+,\phi^-}:=\pare{f,g}_{{t}}
				&:=\sum_{{\mu}} \left[ \pare{\tilde{\zeta}_{{\mu}}\Psi_{{\mu},t}^+\zeta_{{\mu}} f^{{\mu}},\tilde{\zeta}_{{\mu}}\Psi_{{\mu},t}^+\zeta_{{\mu}} g^{{\mu}}}_{\phi^+} + \pare{\tilde{\zeta}_{{\mu}}\Psi_{{\mu},t}^0\zeta_{{\mu}} f^{{\mu}},\tilde{\zeta}_{{\mu}}\Psi_{{\mu},t}^0\zeta_{{\mu}} g^{{\mu}}}_{0} \right. \\
				&\left. +\pare{\tilde{\zeta}_{{\mu}}\Psi_{{\mu},t}^-\zeta_{{\mu}} f^{{\mu}},\tilde{\zeta}_{{\mu}}\Psi_{{\mu},t}^-\zeta_{{\mu}} g^{{\mu}}}_{\phi^-} \right]
\end{align*}
and
\[
\normm{f}^2_{\phi^+,\phi^-}:=\sum_{{\mu}} \corc{ \norm{\tilde{\zeta}_{{\mu}}\Psi_{{\mu},t}^+\zeta_{{\mu}} f^{{\mu}}}^2_{\phi^+}
+\norm{\tilde{\zeta}_{{\mu}}\Psi_{{\mu},t}^0\zeta_{{\mu}} f^{{\mu}}}^2_{0} + \norm{\tilde{\zeta}_{{\mu}}\Psi_{{\mu},t}^-\zeta_{{\mu}} f^{{\mu}}}^2_{\phi^-}}
\]
where $ f^{{\mu}} $ and $g^{\mu}$ are the  forms $ f $  and $g$, respectively, expressed in the local coordinates on $ U_\mu$. 
The superscript $\mu$ will often omitted. {In the case
that $\phi^+(z) =  t|z|^2$ or $-t|z|^2$ and $\phi^-(z) = -t|z|^2$ or $t|z|^2$, we denote the norm by $\normm{\cdot}_t$ and in general replace the subscript with
$t$ (e.g., we write $c_t$ for $c_{\phi^+,\phi^-}$).}

For a form $ f $ on $ M $,  the Sobolev norm of order $ s $ is given by  the following:
\[
\norm{f}_{H^s}^2=\sum_{{\mu}}\norm{\tilde{\zeta}_{{\mu}}\Lambda^s\zeta_{{\mu}} f^{{\mu}}}_0^2
\]
where $ \Lambda $ is the pseudodifferential operator with symbol $ (1+\vab{\xi}^2)^{1/2} $. 
In \cite{Nic06}, Nicoara shows that there exist constants $ c_{{\phi^+,\phi^-}} $ and $ C_{{\phi^+,\phi^-}} $ so that
\begin{equation}\label{eq:normequiv}
c_{{\phi^+,\phi^-}}\norm{f}_0^2\leq \normm{f}^2_{\phi^+,\phi^-}\leq C_{{\phi^+,\phi^-}}\norm{f}_0^2.
\end{equation}
Additionally, there exists a invertible self-adjoint operator $ E_{{\phi^+,\phi^-}} $ 
so that $ \pare{f,g}_0=\pare{f,E_{{\phi^+,\phi^-}} g}_{{\phi^+,\phi^-}} $, where $ E_{{\phi^+,\phi^-}} $ is the inverse of 
\[
\sum_{{\mu}} \pare{\zeta_{{\mu}}(\Psi_{{\mu},t}^+)^*\tilde{\zeta}_{{\mu}} e^{-\phi^+}\tilde{\zeta}_{{\mu}}\Psi_{{\mu},t}^+\zeta_{{\mu}} + \zeta_{{\mu}}(\Psi_{{\mu},t}^0)^*\tilde{\zeta}_{{\mu}}^2 \Psi_{{\mu},t}^0\zeta_{{\mu}} + \zeta_{{\mu}}(\Psi_{{\mu},t}^-)^*\tilde{\zeta}_{{\mu}} e^{-\phi^-}\tilde{\zeta}_{{\mu}}\Psi_{{\mu},t}^-\zeta_{{\mu}}}
\]
and {this operator is} bounded in $ L^2(M) $ independently of $ tA\geq 1 $ (see Corollary 4.6 in \cite{Nic06}).

\subsection{$ \pbb $ and its adjoints}

If $ f $ is a function on $ M $, {then} in a local coordinates 
\[ 
\pbb f=\sum_{j=1}^{n-1}\lb_j f\,\omb_j 
\]
and if $ f=\sum_{J\in \I_q} f_J\,\omb^J$ is a $(0,q)$-form, then there exist functions $ m_K^J $ such that
\[
\pbb f=\sum_{J\in \I_q, K\in \I_{q+1}}\sum_{j=1}^{n-1}\epsilon_{K}^{jJ}\lb_jf_J\,\omb^K+ \sum_{J\in \I_q,K\in \I_{q+1}}f_J m_K^J\,\omb^K
\]
 where $ \epsilon_{K}^{jJ}$ is equal to 0 if $ \lla{K}\neq\lla{j}\cup J $ and is the sign of the permutation that reorders $ jJ $ to $ K $ {otherwise}. We also define
\begin{equation}\label{eq:subindex}
f_{jI}=\sum_{J\in \I_q}\epsilon_{J}^{jI}f_J
\end{equation}
(in this case, $I\in\I_{q-1}$). Let $ \lb_j^{*} $ be the adjoint of $ \lb_j $ in $ \pare{\ ,\ }_0 $, $ \lba_j $ be the adjoint of $ \lb_j $ in $ \pare{\ ,\ }_\phi$. 
Then on a small neighborhood $ U $ we will have $ \lb_j^{*}=-L_j+\sigma_j $ and $ \lba_j = -L_j+L_j\phi+\sigma_j$ where 
$ \sigma_j $ is smooth function on $ U $. Because we will need it later, we observe that there are smooth functions $d_{sr}^\ell$ and $\sigma_s$ so that
\begin{equation}\label{eqn:L, bar L commutator}
\corc{ \lb_r,\lba_s } = c_{sr}T+ \lb_rL_s\phi + \sum_{\ell=1}^{n-1}(d_{sr}^\ell L_\ell-\bar{d}^\ell_{rs}\lb_\ell) + \lb_r\sigma_s.
\end{equation}

{We denote the $L^2$ adjoint of $\dbarb$ in $ L^2_{0,q}(M,e^{-\phi})$} by $ \pbb^{*,\phi} $. 
For the remainder of the paper, $ \phi $ stands for either $ \phi^{+} $ or $ \phi^- $ {and
\[
|\phi^+(z)| = |\phi^-(z)| = |t| |z|^2,
\]
though virtually all of our calculations hold for general $\phi$, up to the point when our calculation require an analysis of the eigenvalues of the Levi form.}

To keep track of the terms that arise in our integration by parts, we use the following shorthand 
for forms $f$ supported in a neighborhood $U_{{\mu}}$ (recognizing that these operators depend on
our choice of neighborhoods $\{U_{{\mu}}\}$):
\begin{align*}
\nabla_{{\bar L^{*,\phi}}} f &=\sum_{J\in\I_q}\sum_{j=1}^{n-1}\lba_j f_J\, \omb^J;
& \norm{\bna_\Upsilon f}_\phi^2&=\sum_{J\in\I_q}\sum_{j,k=1}^{n-1}\pare{b^{\ba{k}j}\lb_kf_J,\lb_jf_J}_\phi 
:= \sum_{j,k=1}^{n-1}\pare{b^{\ba{k}j}\lb_kf,\lb_jf}_\phi \\
\nabla_{\bar L} f &=\sum_{J\in\I_q}\sum_{j=1}^{n-1}\lb_jf_J\, \omb^J;
& \norm{{\nabla}_\Upsilon f}_\phi^2&=\sum_{J\in\I_q}\sum_{j,k=1}^{n-1}\pare{\bkj \lba_j f_J,\lba_k f_J}_\phi
:= \sum_{j,k=1}^{n-1}\pare{\bkj \lba_j f,\lba_k f}_\phi
\end{align*}
Again, if $ f=\sum_{J\in \I_q} f_J \,\omb^J$ is defined locally, then
\begin{align*}
\pbb^* f&=\sum_{I\in \I_{q-1},J\in I_{q}}\sum_{j=1}^{n-1}\epsilon_J^{jI}\lb_j^*f_J\,\omb^I+ \sum_{I\in \I_{q-1},J\in I_{q}}f_J\ba{m}_J^I \,\omb^I\\
&=\sum_{I\in \I_{q-1}}\sum_{j=1}^{n-1}\lb_j^*f_{jI}\,\omb^I+\sum_{I\in \I_{q-1},J\in I_{q}}f_J\ba{m}_J^I \,\omb^I 
\end{align*}
and
\begin{align*}
	\pbb^{*,\phi}f&=\sum_{I\in \I_{q-1}}\sum_{j=1}^{n-1}\lba_jf_{jI}\,\omb^I+\sum_{I\in \I_{q-1},J\in \I_q}f_J\ba{m}_J^I \,\omb^I
\end{align*}

Note that a consequence of the compactness of $M$ and the boundedness of $\phi$, the domains of $ \pbba $ 
and $ \pbb^{*,\phi} $ are equal. Also we have $ \pbb^{*,\phi}=\pbba-\corc{\pbba,\phi} $. 
Let $ \pb_{b,{t}}^* $ be the adjoint of $ \pbb $ with respect to the inner product ${(\cdot,\cdot)_t} $.
{We also define the weighted Kohn Laplacian $\Boxb$ by
$ \Box_{b,t}:=\pbb\dbars_{b,{t}} +\dbars_{b,{t}}\pbb$ where
\begin{eqnarray*}
\Dom(\Box_{b,t}):=\lla{ \phi\in L^2_{0,q}(M) :\phi\in \Dom(\pbb)\cap \Dom(\dbars_{b,t}),
\ \pbb\phi\in \Dom(\dbars_{b,t}),\ and\ \dbars_{b,t}\phi\in \Dom(\pbb)}.
\end{eqnarray*}
}

The computations proving Lemmas 4.8 and 4.9 and equation (4.4) in \cite{Nic06} can be applied here with only a change of notation, so we have the following two results, recorded here as Lemmas \ref{lemma:nic1} and \ref{lemma:nic2}. The {consequence} is that 
$ {\pb_{b,{t}}^*} $ acts like {$ \pbb^{*,\phi^+} $ (denoted just by $  \pbb^{*,+} $)} for forms whose support is basically 
$ \opC^+ $ and {$ \pbb^{*,\phi^-} $ (denoted just by $ \pbb^{*,-}$)} on forms whose support is basically $ \opC^- $.

\begin{lemma}\label{lemma:nic1}
On smooth (0,q)-forms,
\begin{eqnarray*}
\pb_{b,{t}}^* &=&\pbba-\sum_\mu \zeta_\mu^2\tilde{\Psi}_{\mu,t}^+\corc{\pbba,\phi^+}+\sum_\mu \zeta_\mu^2\tilde{\Psi}_{\mu,t}^-\corc{\pbba,\phi^-}\\
	&&+\sum_\mu\left( \tilde{\zeta}_\mu\corc{\tilde{\zeta}_\mu \Psi_{\mu,t}^+\zeta_\mu,\pbb}^*\tilde{\zeta}_\mu \Psi_{\mu,t}^+\zeta_\mu+ \zeta_\mu(\Psi_{\mu,t}^+)^*\tilde{\zeta}_\mu\corc{ \pb_b^{*,+},\tilde{\zeta}_\mu\Psi_{\mu,t}^+\zeta_\mu }\tilde{\zeta}_\mu \right.\\
	&&+\left. \tilde{\zeta}_\mu\corc{\tilde{\zeta}_\mu \Psi_{\mu,t}^-\zeta_\mu,\pbb}^*\tilde{\zeta}_\mu \Psi_{\mu,t}^-\zeta_\mu+ \zeta_\mu(\Psi_{\mu,t}^-)^*\tilde{\zeta}_\mu\corc{ \pb_b^{*,-},\tilde{\zeta}_\mu\Psi_{\mu,t}^-\zeta_\mu }\tilde{\zeta}_\mu + E_A  \right)
\end{eqnarray*}
where the error term $ E_A $ is a sum of order zero terms and ``lower order" terms. Also, the symbol of $ E_A $ is supported in $ \opC_\mu^0 $ for each $ \mu $.
\end{lemma}
We use the following energy forms in our calculations:
\begin{align*}
Q_{b,{t}}(f,g)&=\pare{\pbb f,\pbb g}_{{t}}+\pare{\pb_{b,{t}}^* f,\pb_{b,{t}}^* g}_{{t}}\\
Q_{b,+}(f,g)&=\pare{\pbb f,\pbb g}_{\phi^+}+\pare{\pb_{b}^{*,+} f,\pb_{b}^{*,+} g}_{\phi^+}\\
Q_{b,0}(f,g)&=\pare{\pbb f,\pbb g}_{0}+\pare{\pb_{b}^{*} f,\pb_{b}^{*} g}_{0}\\
Q_{b,-}(f,g)&=\pare{\pbb f,\pbb g}_{\phi^-}+\pare{\pb_{b}^{*,-} f,\pb_{b}^{*,-} g}_{\phi^-}.
\end{align*}
{The space of weighted harmonic forms $ \opH_{{t}}^q $ is defined by
\begin{align*}
\opH_t^q&:=\lla{f\in \Dom (\pbb)\cap \Dom(\pbba): \pbb f=0,\pb_{b,t}^*f=0} \\
&=\lla{f\in \Dom (\pbb)\cap \Dom(\pbba): Q_{b,t}(f,f)=0}.
\end{align*}
}

We have the following relationship between the energy forms. See \cite[Lemma 3.4]{HaRa11} or \cite[Lemma 4.9]{Nic06}.
\begin{lemma}\label{lemma:nic2}
If $ f $ is a smooth (0,q)-form on $ M $, then there exist constants $ K,K_{{t}} $ and $ K' $ with $ K\geq 1 $ so that
\begin{eqnarray*}
KQ_{b,{t}}(f,f)&+&K_{{t}}\sum_{\nu}\norm{\tilde{\zeta}_{{\mu}}\tilde{\Psi}_{{\mu},t}^0\zeta_{{\mu}} f^{{\mu}}}_0^2+K'\normm{f}_{{t}}^2+O_t(\norm{f}_{-1}^2) \\
&\geq& \sum_{{\mu}} \left[ Q_{b,+}(\tilde{\zeta}_{{\mu}}\Psi_{{\mu},t}^+\zeta_{{\mu}} f^{{\mu}}, \tilde{\zeta}_{{\mu}}\Psi_{{\mu},t}^+\zeta_{{\mu}} f^{{\mu}})   \right.\\
&&\left. Q_{b,0}(\tilde{\zeta}_{{\mu}}\Psi_{{\mu},t}^0\zeta_{{\mu}} f^{{\mu}}, \tilde{\zeta}_{{\mu}}\Psi_{{\mu},t}^0\zeta_{{\mu}} f^{{\mu}})+Q_{b,-}(\tilde{\zeta}_{{\mu}}\Psi_{{\mu},t}^-\zeta_{{\mu}} f^{{\mu}}, \tilde{\zeta}_{{\mu}}\Psi_{{\mu},t}^-\zeta_{{\mu}} f^{{\mu}}) \right]
\end{eqnarray*}

$ K $ and $ K' $ do not depend on $ t,\phi^- $ or $ \phi^+ $.
\end{lemma}

%
%
\section{The Basic Estimate}\label{sec: basic estimate}

In this section, we compile the technical pieces that will allows us to establish a basic estimate the ground level $L^2$ estimates
for Theorem \ref{thm:mainthm, Sobolev} in Section \ref{sec:main theorem, weighted}.
\begin{proposition}\label{prop:mainestimate}
Let $ M^{2n-1} \subset \C^N$ be a smooth, compact, orientable CR-manifold of hypersurface type that satisfies weak $Y(q)$ for some fixed
$ 1\leq q\leq n-2 $.  Set
\begin{align} \label{eqn:phi+, phi- defns}
\phi^+(z) &= 
\begin{cases}
t|z|^2      & \text{ if } \omega(\Upsilon_q)< q \\
-t|z|^2 &  \text{ if } \omega(\Upsilon_q)> q
\end{cases}
&\text{and}&  &
\phi^-(z) &=
\begin{cases}
-t|z|^2      & \text{ if } \omega(\Upsilon_{n-1-q})< n-1-q \\
t|z|^2 &  \text{ if } \omega(\Upsilon_{n-1-q})> n-1-q.
\end{cases}
\end{align}
There exist constants $ K $ and $ K_{{t}} $ where $ K $ does not depend on $t$ so that 
\begin{equation}\label{estimateconclusion}
t\normm{f}_{{t}}^2\leq KQ_{b,{t}}(f,f) +K_{{t}}\norm{f}_{-1}^2{,}
\end{equation}
for $ t $ sufficiently large.
\end{proposition}

The main work in establishing \eqref{estimateconclusion} is {to prove} the following:
\begin{equation}\label{eq:mainestimative}
t\normm{f}^2_{{t}} \leq KQ_{b,{t}}(f,f)+K\normm{f}_{{t}}^2+K_{{t}}\sum_{{\mu}}\sum_{J\in \I_q}\norm{\tilde{\zeta}_{{\mu}}  \tilde{\Psi}_{{\mu},t}^0\zeta_{{\mu}} f_J^{{\mu}}}_0^2 + K_{{t}} ' \norm{f}_{-1}^2.
\end{equation}

In order to prove \eqref{eq:mainestimative}, 
we {estimate} a $(0,q)$-form $f$ with support in neighborhood $U$ in a generic energy form
$Q_{b,\phi}(f,g) := (\dbarb f, \dbarb g)_\phi + (\dbarb^{*,\phi}f,\dbarb^{*,\phi}g)_\phi$. Throughout the estimate, we will make use of three 
terms, $E_0(f)$, $\tilde E_1(f)$, and $\tilde E_2(f)$ to collect the error terms that we will bound later. {We want}
$E_0(f) = O(\norm{f}_\phi^2)$ {and}
\begin{align*}
\tilde E_1(f) &=  \sum_{J,J'\in \I_q}\sum_{j=1}^{n-1} \pare{  \lb_j f_J, a_{JJ'} f_{J'}}_\phi
&\text{and}& &
\tilde E_2(f) &=\sum_{J,J'\in \I_q}\sum_{j=1}^{n-1} \pare{ \lba_j f_J, \tilde a_{JJ'} f_{J'} }_\phi
\end{align*}
for some collection of smooth functions $a_{JJ'}$ and $\tilde a_{JJ'}$ {that may change line to line}.

Integration by parts (see, e.g., \cite[Lemma 4.2]{Rai10}) shows that
\begin{align*}
Q_{b,\phi}(f,f)
&=\norm{\nabla_{\bar L}f}_\phi^2+\sum_{J,J'\in I_{q}}\sum_{\atopp{j,k=1}{j\neq k}}^{n-1}  \epsilon_{jJ'}^{kJ} \pare{ \corc{\lba_j,\lb_k}f_J,f_{J'} }_\phi \\
&+\sum_{J\in \I_q}\sum_{j\in J} \pare{\corc{\lb_j,\lba_j}f_J,f_J}_\phi  +2\Rre \pare{\tilde{E}_2(f)+\tilde{E}_1(f)}+E_0(f).
\end{align*}
{Developing} the commutator terms as in \cite[Lemma 4.2]{Rai10} and using the
fact that $L_j =-\lba_j+L_j\phi+\sigma_j $, we have the equality
\begin{align*}
Q_{b,\phi}(f,f)&=\norm{\nabla_{\bar L}f}_\phi^2 +\sum_{I\in \I_{q-1}}\sum_{j,k=1}^{n-1}\Rre\pare{c_{jk}T f_{jI},f_{kI} }_\phi \\
&+\Rre\sum_{I\in \I_{q-1}}\sum_{j,k=1}^{n-1}\corc{\big((\lb_k L_j\phi) f_{jJ},f_{kI}\big)_\phi +\pare{\sum_{l=1}^{n-1}d_{jk}^lL_l\phi f_{jI},f_{kI}  }_\phi  }\\
& +{\tilde{E}_1(f)+\tilde{E}_2(f)+E_0(f).}
\end{align*}

Since
\begin{align}
\Rre \sum_{I\in \I_{q-1}}\sum_{j,k=1}^{n-1} \pare{\lb_k L_j\phi f_{jJ},f_{kI}}_\phi&=\frac{1}{2} \sum_{I\in \I_{q-1}}\sum_{j,k=1}^{n-1} \pare{(\lb_k L_j\phi+L_j\lb_k\phi) f_{jJ},f_{kI}}_\phi\nn \\
\Rre \sum_{I\in \I_{q-1}}\sum_{j,k=1}^{n-1}\pare{\sum_{l=1}^{n-1}d_{jk}^lL_l\phi f_{jI},f_{kI}  }_\phi&=\frac{1}{2} \sum_{I\in \I_{q-1}}\sum_{j,k=1}^{n-1} \pare{\sum_{l=1}^{n-1}(d_{jk}^lL_l \phi+\bar{d}_{kj}^l\lb_l\phi) f_{jJ},f_{kI}}_\phi 
\label{eqn:real is sum of conjugates}
\end{align}
and
\[
\frac{1}{2}\pare{\lb_k L_j\phi+L_j\lb_k\phi}+\frac{1}{2}\sum_{l=1}^{n-1}(d_{jk}^lL_l \phi+\bar{d}_{kj}^l\lb_l\phi)= \Theta_{jk}^\phi-\frac{1}{2}\nu(\phi)c_{jk}
\]
it follows that
\begin{align}
Q_{b,\phi}(f,f)&=\norm{\nabla_{\bar L}f}_\phi^2 +\sum_{I\in \I_{q-1}}\sum_{j,k=1}^{n-1}\Rre\pare{c_{jk}T f_{jI},f_{kI} }_\phi \nonumber \\
&+\sum_{I\in \I_{q-1}}\sum_{j,k=1}^{n-1}\pare{ (\Theta_{jk}^\phi -\frac{1}{2}\nu(\phi)c_{jk})  f_{jI},f_{kI}}_\phi  \label{eq:MKH1} 
+\tilde{E}_1(f)+\tilde{E}_2(f)+E_0(f). 
\end{align}

On the other hand, integration by parts, expanding the commutator terms, and using \eqref{eqn:real is sum of conjugates}, we will have
\begin{align}
\norm{\bna_\Upsilon f}^2_\phi
&=\sum_{j,k=1}^{n-1}\corc{\pare{\bkj\lba_j f,\lba_k f}_\phi+ \pare{\corc{\lba_j,\lb_k} f,\bjk f}_\phi +\pare{\lba_j(\bkj)\lb_k f,f}_\phi}\nonumber\\
&+\sum_{j,k=1}^{n-1}\pare{\lba_j f,\lba_k(\bjk)f}_\phi\nonumber\\
&= \| \nabla_\Upsilon f \|_\phi^2
-\sum_{j,k=1}^{n-1}\corc{\pare{\bkj c_{jk}Tf,f}_\phi+\pare{\bkj(\Theta_{jk}^\phi-\frac{1}{2}\nu(\phi)c_{jk} )f,f}_\phi}  \label{eq:DWIBP1}\\
& {+\tilde{E}_2(f)+\tilde{E}_1(f)+E_0(f).} \nn
\end{align}
Motivated by \cite[p.1725]{HaRa15}, we write 
$
\norm{\nabla_{\bar L} f}_\phi^2=\pare{\norm{\nabla_{\bar L} f}_\phi^2-\norm{\bna_\Upsilon f}_\phi^2}+\norm{\bna_\Upsilon f}_\phi^2
$ 
and {use \eqref{eq:DWIBP1} to obtain}
\begin{align*}
Q_{b,\phi}(f,f)
&=\pare{\norm{\nabla_{\bar L}f}_\phi^2-\norm{\bna_\Upsilon f}_\phi^2 }+\norm{\nabla_\Upsilon f}_\phi^2+\sum_{I\in \I_{q-1}}\sum_{j,k=1}^{n-1}\Rre\pare{c_{jk}T f_{jI},f_{kI} }_\phi\\
&-\pare{ i\intprod{d\gamma}{\Upsilon}Tf,f}_\phi+\sum_{I\in \I_{q-1}}\sum_{j,k=1}^{n-1}\pare{ (\Theta_{jk}^\phi -\frac{1}{2}\nu(\phi)c_{jk})  f_{jJ},f_{kI}}_\phi\\
&-\pare{ i\intprod{ \Theta^\phi }{\Upsilon}f,f}_\phi+\pare{\frac{1}{2}\nu(\phi)i\intprod{d\gamma}{\Upsilon}f,f}_\phi
+\tilde{E}_1(f)+\tilde{E}_2(f)+E_0(f)
\end{align*}
Since
\[
\sum_{J\in \I_q}(a f_J,f_J)_\phi=\sum_{I\in \I_{q-1}}\sum_{j,k=1}^{n-1}\pare{\frac{a\delta_{jk}}{q}f_{jI},f_{kI} }_\phi
\]
where $ (\delta_{jk}) $ is the identity matrix $I_{n-1}$, we have
\begin{align*}
&Q_{b,\phi}(f,f)=\pare{\norm{\nabla_{\bar L}f}_\phi^2-\norm{\bna_\Upsilon f}_\phi^2 }+\norm{\nabla_\Upsilon f}_\phi^2\\
&+\sum_{I\in \I_{q-1}}\sum_{j,k=1}^{n-1}\Rre\pare{\pare{c_{jk}-\frac{i\intprod{d\gamma}{\Upsilon}\delta_{jk}}{q}}T f_{jI},f_{kI} }_\phi
+\sum_{I\in \I_{q-1}}\sum_{j,k=1}^{n-1}\pare{\pare{\Theta_{jk}^\phi-\frac{i\intprod{\Theta}{\Upsilon}\delta_{jk}}{q}} f_{jI},f_{kI} }_\phi\\
&-\sum_{I\in \I_{q-1}}\sum_{j,k=1}^{n-1}\pare{\frac{1}{2}\nu(\phi){\pare{c_{jk}-\frac{i\intprod{d\gamma}{\Upsilon}\delta_{jk}}{q}} }f_{jI},f_{kI} }_\phi
+\tilde{E}_1(f)+\tilde{E}_2(f)+E_0(f).
\end{align*}
Bounding the error terms $ \tilde{E}_1(f) $ and $\tilde{E}_2(f) $ uses the same argument, and we demonstrate the bound
for $\tilde{E}_1(f)$. 
Terms of the form $\sum_{j=1}^{n-1}\pare{a_j\lb_j g,h}_\phi $ comprise $\tilde E_1$ for various functions $g$ and $h$, and we compute
\begin{equation}\label{eqn:decomposing E_1}
\sum_{j=1}^{n-1}\pare{a_j\lb_j g,h}_\phi=\sum_{j,k=1}^{n-1}\pare{(\delta_{jk}-\bjk)\lb_j g,\bar{a}_kh}_\phi+\sum_{j,k=1}^{n-1}\pare{ \bjk \lb_j g,\bar{a}_k h }_\phi.
\end{equation}
To estimate the first terms, observe that for $\ve>0$, a small constant/large constant argument shows that
\[
\vab{ \sum_{j,k=1}^{n-1}\pare{(\delta_{jk}-\bjk)\lb_j g,\bar{a}_kh}_\phi }
\leq \ve\sum_{k=1}^{n-1}\Big\|\sum_{j=1}^{n-1}(\delta_{jk}-\bjk)\lb_j g \Big\|_\phi^2+ O_{\frac{1}{\ve}}(\norm{h}_\phi^2).
\]
Stepping away from the integration (momentarily), suppose that at some point in $U$, 
$ A $ is a unitary matrix that diagonalizes the hermitian matrix $\bar B = (b^{\bar jk})$ of $\Upsilon$ such that 
$ \bar{B}=A^*\Lambda A $, where $ \Lambda=\diag\lla{\lambda_1,\dots,\lambda_{n-1}} $ and  $ \lambda_1, \cdots , \lambda_{n-1} $ 
are the eigenvalues of $ \bar{B} $. Consider $[\bar L_j g]$ as a column vector with components $[\bar L_j g]_k$. 
Then since $(1-\lambda_j)^2 \leq (1-\lambda_j)$ for all $j$, 
\begin{multline*}
\sum_{k=1}^{n-1}\vab{\sum_{j=1}^{n-1}(\delta_{jk}-\bjk)(\lb_j g)}^2
=\vab{\corc{Id-B}\corc{\lb_j g}}^2 
=\sum_{j=1}^{n-1}(1-\lambda_j)^2\vab{ \corc{A \corc{\overline{ \bar{L}_jg }} }_j  }^2\\
\leq \sum_{j=1}^{n-1}(1-\lambda_j)\vab{ \corc{A \corc{\overline{ \bar{L}_jg }}}_j  }^2 
=\sum_{j=1}^{n-1}\vab{\lb_j g}^2-{\sum_{j,k=1}^{n-1}}\bkj \overline{\lb_j g}\lb_k g.
\end{multline*}
%
Returning to the integration, we now observe,
\[
\sum_{k=1}^{n-1}\Big\|\sum_{j=1}^{n-1}(\delta_{jk}-\bjk)\lb_jf\Big\|_\phi^2\leq\norm{\nabla_{\bar L} f}_\phi^2-\norm{\bna_\Upsilon f}_\phi.
\]
For the second term in (\ref{eqn:decomposing E_1}), a similar small constant/large constant argument shows
\[
\vab{\sum_{j,k}(a_kf,\bkj \lba_j g)_\phi}\leq O_{\frac{1}{\ve}}(\norm{f}_\phi^2)+\ve\sum_{k=1}^{n-1}\Big\|\sum_{j=1}^{n-1}\bkj \lba_j g  \Big\|_\phi^2,
\]
and linear algebra (as above) helps to establish
\[
\sum_{k=1}^{n-1}\Big\| \sum_{j=1}^{n-1} \bkj \lba_jg \Big\|_\phi^2\leq \sum_{j,k}\pare{\bkj\lba_j g,\lba_k g}_\phi=\norm{\nabla_\Upsilon g}_\phi^2.
\]
\\
Summarizing the above, for $ \ve $ sufficiently small {and $f$ supported in a small neighborhood}, we have
\begin{eqnarray}
	Q_{b,\phi}(f,f)&\geq&\sum_{I\in \I_{q-1}}\sum_{j,k=1}^{n-1}\Rre\pare{\pare{c_{jk}-\frac{i\intprod{d\gamma}{\Upsilon}\delta_{jk}}{q}}T f_{jI},f_{kI} }_\phi\nonumber \\
	&&+\sum_{I\in \I_{q-1}}\sum_{j,k=1}^{n-1}\pare{\pare{\Theta_{jk}^\phi-\frac{i\intprod{\Theta^{\phi}}{\Upsilon}\delta_{jk}}{q}} f_{jI},f_{kI} }_\phi \nonumber\\
	&&-\sum_{I\in \I_{q-1}}\sum_{j,k=1}^{n-1}\pare{\frac{1}{2}\nu(\phi)\pare{\pare{c_{jk}-\frac{i\intprod{d\gamma}{\Upsilon}\delta_{jk}}{q}} }f_{jI},f_{kI} }_\phi+ O(\norm{f}_\phi^2)  \label{eq:maininequality1}
\end{eqnarray}

{To handle the $T$ terms, we recall the following results.}
The first is a well-known multilinear algebra result that appears (among other places) in Straube \cite{Str10}:
\begin{lemma}\label{lemma:StraubeAlgebra}
	Let $B= \pare{b_{jk}}_{1\leq j,k\leq n-1} $ be a Hermitian matrix and $ 1\leq q\leq n-1$. The following are  equivalent:
	\begin{enumerate}[i.]
		\item If $ u\in \Lambda^{0,q} $, then $ \sum_{K\in \I_{q-1}}\sum_{j,k=1}^{n-1} b_{jk}u_{jK}\overline{u_{kK}}\geq M\vab{u}^2 $.\\
		\item The sum of any q eigenvalues of $ B $ is at least $ M $.
		\item $ \sum_{s=1}^q\sum_{j,k=1}^{n-1} b_{jk}t_j^s\overline{t_k^s}\geq M $ for any orthonormal vectors $ \lla{t^s}_{1\leq s\leq q} \subset \C^{n-1}$.
	\end{enumerate}
\end{lemma}
The next two results  are consequences of the sharp G{\aa}rding Inequality
and appear as \cite[Lemma 4.6, Lemma 4.7]{Rai10}.
\begin{lemma}\label{prop:Rai10badT}
	Let $ f $ a (0,q)-form supported on $ U $ so that up to a smooth term $ \hat{f} $ is supported in $ \opC^+ $, and let $ \corc{h_{jk}} $ a Hermitian matrix such that the sum of any q eigenvalues is $ \geq 0 $. Then
\begin{eqnarray*}
\Rre \Big\{\sum_{I\in \I_{q-1}}\sum_{j,k=1}^{n-1}\pare{h_{jk}T f_{jI},f_{kI}}_\phi \Big\}
\geq tA \Rre \sum_{I\in \I_{q-1}}\sum_{j,k=1}^{n-1}\pare{h_{jk}f_{jI},f_{kI}}_\phi -O(\norm{f}_\phi^2)-O_t(\norm{\tilde\zeta \widetilde\Psi_t^0f}_0^2).
\end{eqnarray*}
\end{lemma}

\begin{lemma}
	Let $ f $ a (0,q)-form supported on $ U $ so that up to a smooth term $ \hat{f} $ is supported in $ \opC^- $, and let $ \corc{h_{jk}} $ a Hermitian matrix such that the sum of any n-1-q eigenvalues is $ \geq 0 $. Then
\begin{multline*}
\Rre \lla{\sum_{J\in \I_{q}}\sum_{j=1}^{n-1}\pare{h_{jj}(-T) f_{J},f_{J}}_\phi - \sum_{I\in \I_{q-1}}\sum_{j,k}\pare{h_{jk}(-T) f_{jI},f_{kI}}_\phi } \\
\geq tA \Rre \lla{\sum_{J\in\I_{q}}\sum_{j=1}^{n-1}\pare{h_{jj} f_{J},f_{J}}_\phi - \sum_{I\in \I_{q-1}}\sum_{j,k}\pare{h_{jk} f_{jI},f_{kI}}_\phi } 
-O(\norm{f}_\phi^2)-O_t(\norm{\tilde\zeta \tilde\Psi_t^0f}_0^2).
\end{multline*}
\end{lemma}

Now, we are ready to estimate $ Q_{b,+}(\cdot,\cdot) $ and $ Q_{b,-}(\cdot,\cdot) $. 
\begin{proposition}\label{prop:estim+}
Let $ f \in \Dom \pbb\cap \Dom \pbba$ be a $(0,q)$-form supported in $ U $
and let $\phi$ be as in (\ref{eqn:phi+, phi- defns}). Then there exists a constant $C$  so that
\begin{equation*}
Q_{b,+}\pare{\tilde\zeta\Psi_{t}^+f, \tilde\zeta\Psi_{t}^+f }+C\norm{\tilde\zeta\Psi_{t}^+f}_{\phi^+}+O_t(\norm{\tilde\zeta\tilde{\Psi}_{t}^0f}_0^2)
\geq tB_{\phi^+}\norm{\tilde\zeta\Psi_{t}^+f}^2_{\phi^+}.
\end{equation*}
\end{proposition}

\begin{proof}
By  \eqref{eq:maininequality1}, the fact that  the Fourier transform of {$\tilde\zeta \Psi_t^+f $}
is supported in $\opC^+$ up to smooth term, and Proposition \ref{prop:Rai10badT}, we  have
\begin{align*}
&Q_{b,+}(\tilde\zeta\Psi_t^+f,\tilde\zeta\Psi_t^+f)
\geq tA\sum_{I\in \I_{q-1}}\sum_{j,k=1}^{n-1}\Rre\pare{\pare{c_{jk}-\frac{i\intprod{d\gamma}{\Upsilon}\delta_{jk}}{q}} \tilde\zeta\Psi_t^+f_{jI},\tilde\zeta\Psi_t^+f_{kI} }_{\phi^+}\nonumber \\
&+\sum_{I\in \I_{q-1}}\sum_{j,k=1}^{n-1}\pare{\pare{\Theta_{jk}^{\phi+}-\frac{i\intprod{\Theta^{\phi+}}{\Upsilon}\delta_{jk}}{q}} \tilde\zeta\Psi_t^+f_{jI},\tilde\zeta\Psi_t^+f_{kI} }_{\phi^+} \nonumber\\
&-\sum_{I\in \I_{q-1}}\sum_{j,k=1}^{n-1}\pare{\frac{1}{2}\nu(\phi^+)\pare{\pare{c_{jk}-\frac{i\intprod{d\gamma}{\Upsilon}\delta_{jk}}{q}} }\tilde\zeta\Psi_t^+f_{jI},\tilde\zeta\Psi_t^+f_{kI} }_{\phi^+}\\
& - O(\norm{\tilde\zeta\Psi_t^+f}_{\phi^+}^2)-O_t(\norm{\tilde\zeta \tilde{\Psi}_t^0f}_0^2)
\end{align*}
By choosing $ A\geq {\sup_{z \in M}}\frac{1}{2}\vab{\nu(\vab{z}^2)} $, Lemma \ref{lemma:StraubeAlgebra} implies that
\[
Q_{b,+}(\tilde\zeta\Psi_t^+f,\tilde\zeta\Psi_t^+f) + C\norm{\tilde\zeta\Psi_t^+f}_{\phi^+}^2+O_t(\norm{\tilde\zeta \tilde{\Psi}_{{t}}^0f}_0^2)\geq t{B_{\phi^+}}\norm{\tilde\zeta\Psi_t^+f}_{\phi^+}^2
\]
for some constants $ C $ and $ B_{\phi^+} $ {where $B_{\phi^+}$ satisfies $ \vab{q-\omega(\Upsilon)}>B_{\phi^+} $ on $ M $}
\end{proof}


In order to estimate the terms $ Q_{b,-}(\tilde\zeta\Psi_t^-f,\tilde\zeta\Psi_t^-f) $ 
we have to modify the analysis slightly from the $Q_{b,+}$ case. Similarly to \eqref{eq:MKH1}, we have 
\begin{align}\label{eq:MKH2}
&Q_{b,\phi}(f,f)
=\norm{\nabla_{{\bar L^{*,\phi}}} f}_\phi^2 +\sum_{I\in \I_{q-1}}\sum_{j,k=1}^{n-1}\pare{c_{jk}T f_{jI},f_{kI} }_\phi
-\sum_{j=1}^{n-1}\pare{c_{jj}Tf,f}_\phi  \nn\\
&+\sum_{I\in \I_{q-1}}\sum_{j,k=1}^{n-1}\pare{ (\Theta_{jk}^\phi -\frac{1}{2}\nu(\phi)c_{jk})  f_{jI},f_{kI}}_{\phi}  
-\sum_{j=1}^{n-1}\pare{(\Theta_{jj}^\phi-\frac{1}{2}\nu(\phi)c_{jj} )f,f}_\phi \nonumber\\
& -O_\epsilon(\norm{\nabla_{{\bar L^{*,\phi}}} f}_\phi^2-\norm{\nabla_\Upsilon f}_\phi^2)
-O_{{\epsilon}}(\norm{\bna_\Upsilon f}_\phi^2)-O_{\frac{1}{\epsilon}}(\norm{f}_\phi^2)-O(\norm{f}_\phi^2).
\end{align}

Analogously to \eqref{eq:DWIBP1}, we have

\begin{eqnarray}\label{eq:DWIBP2}
\norm{{\nabla}_\Upsilon f}^2_\phi&=&\sum_{j,k=1}^{n-1}\corc{\pare{\bkj\lb_k f,\lb_j f}_\phi+ \pare{\bkj c_{jk}Tf,f}_\phi+\pare{\bkj(\Theta_{jk}^\phi-\frac{1}{2}\nu(\phi)c_{jk} )f,f}_\phi}\nonumber\\
&& -O_\epsilon(\norm{\nabla_{{\bar L^{*,\phi}}}f}_\phi^2-\norm{\nabla_\Upsilon f}_\phi^2)-O_\epsilon(\norm{\bna_\Upsilon f}_\phi^2)-O_{\frac{1}{\epsilon}}(\norm{f}_\phi^2)-O(\norm{f}_\phi^2).
\end{eqnarray}
It now follows from \eqref{eq:MKH2} and \eqref{eq:DWIBP2} that
\begin{align}
&Q_{b,\phi}(f,f)
\geq\sum_{I\in \I_{q-1}}\sum_{j,k=1}^{n-1}\Rre\pare{c_{jk}T f_{jI},f_{kI} }_\phi -\Rre\pare{\sum_{j=1}^{n-1}c_{jj}T f,f }_\phi\nonumber
 -O(\norm{f}_\phi^2) \\
&+\Rre\pare{i\intprod{d\gamma}{\Upsilon}T f,f }_\phi+\pare{i\intprod{\Theta^\phi}{\Upsilon}f,f} 
+\sum_{I\in \I_{q-1}}\sum_{j,k=1}^{n-1}\pare{\Theta_{jk}^\phi f_{jI},f_{kI} }_\phi -\pare{\sum_{j=1}^{n-1}\Theta_{jj}^\phi f,f}  \nonumber\\
&-\sum_{I\in \I_{q-1}}\sum_{j,k=1}^{n-1}\pare{\frac{1}{2}\nu(\phi)c_{jk}f_{jI},f_{kI} }_\phi +\pare{\frac{1}{2}\nu(\phi)\sum_{j=1}^{n-1}c_{jj}f,f}   
-\pare{\frac{1}{2}\nu(\phi)i\intprod{d\gamma}{\Upsilon}f,f}.  \label{eq:maininequality2}
\end{align}
If we set
\[
h_{jk}^-=c_{jk}-\delta_{jk}\frac{i\intprod{d\gamma}{\Upsilon}}{n-1-q},
\quad \text{and}\quad
h_{jk}^\Theta=\Theta_{jk}^\phi-\delta_{jk}\frac{i\intprod{\Theta^\phi}{\Upsilon}}{n-1-q}
\]
then we can rewrite \eqref{eq:maininequality2} by
\begin{align*}
&Q_{b,\phi}(f,f)\geq- \Rre\pare{\sum_{j=1}^{n-1}h_{jj}^- Tf,f }+\sum_{I\in \I_{q-1}}\sum_{j,k=1}^{n-1}\Rre\pare{h_{jk}^-Tf_{jI},f_{kI}}\nonumber \\
&-\pare{\sum_{j=1}^{n-1}h_{jj}^\Theta f,f }+\sum_{I\in \I_{q-1}}\sum_{j,k=1}^{n-1}\pare{h_{jk}^\Theta f_{jI},f_{kI}}\nonumber
+\pare{\frac{1}{2}\nu(\phi)\sum_{j=1}^{n-1}h_{jj}^- f,f }-\sum_{I\in \I_{q-1}}\sum_{j,k=1}^{n-1}\pare{\frac{1}{2}\nu(\phi)h_{jk}^-f_{jI},f_{kI}}\nonumber\\
&-O(\norm{f}_\phi^2)
\end{align*}
Since the sum of $q$ eigenvalues of the matrix $ \frac{Tr(H)}{q}Id - H  $ 
is equal to sum of $ (n-1-q) $ eigenvalues of the matrix $ H $, 
we may now proceed as in the proof of  \eqref{prop:estim+} to obtain the following proposition.
\begin{proposition}\label{prop:estim-}
Let $ f \in \Dom \pbb\cap \Dom \pbba$ be a $(0,q)$-form supported in $ U $
and let $\phi$ be as in (\ref{eqn:phi+, phi- defns}). Then there exists a constant $C$  so that
\begin{equation*}
Q_{b,-}\pare{\tilde\zeta\Psi_{t}^-f, \tilde\zeta\Psi_{t}^-f }+C\norm{\tilde\zeta\Psi_{t}^-f}_{\phi^-}
+O_t(\norm{\tilde\zeta\tilde{\Psi}_{t}^0f}_0^2)\geq t B_{\phi^-}\norm{\tilde\zeta\Psi_{t}^-f}^2_{\phi^-}
\end{equation*}
\end{proposition}

In contrast with the estimates in Lemmas \eqref{prop:estim+} and \eqref{prop:estim-} for forms supported on $ \opC^+ $ and $ \opC^- $ up to smooth terms, we have better estimates for forms supported on $ \opC^0 $ up to smooth terms. The next Lemma can be proved like using the same process done in Lemmas 4.17 and Lemma 4.18 on \cite{Nic06}.
\begin{lemma}\label{lem:forms in C_0}
	Let $ f $ be a (0,q)-form supported in $ U_{{\mu}} $ for some $ {\mu} $ such that up to smooth term, $ \hat{f} $ is supported in $ \widetilde{\opC}^0_{{\mu}} $.  There exist positive constants $ C>1 $ and $ \Gamma $ independent of $ t $ for which
\begin{equation}\label{eq:ellipest}
	CQ_{b,{t}}(f,E_{{t}} f)+\Gamma \norm{f}_0^2\geq \norm{f}_1^2
\end{equation}
\end{lemma}

The other term appearing in our main estimate, $ O\big(\norm{\tilde{\zeta} \widetilde{\Psi}_{{t}}^0 \cdot }_0^2\big)$ can be handled 
with \cite[Proposition 4.11]{Rai10}.
\begin{proposition}\label{prop:Psi_0 terms}
For any $ \epsilon>0 $, there exists $ C_{\epsilon,{t}} >0$ so that
\[
\norm{\tilde{\zeta}\Psi_t^0\zeta\vp_0^2 }\leq \epsilon Q_{b,{t}}(\vp,\vp)+C_{\epsilon,{t}}\norm{\vp}^2_{-1}.
\] 
\end{proposition}

We are finally ready to proof Proposition \ref{prop:mainestimate}.
\begin{proof}[Proof of the Proposition \ref{prop:mainestimate}]
We only need to set the value of the constant $ K, K' $ and $ K_{{t}} $ in Lemma \ref{lemma:nic2}
according to the Propositions \ref{prop:estim+} and \ref{prop:estim-}. From the definition of
$ \normm{\cdot}_{{t}} $, the estimate \eqref{eq:mainestimative} follows. 

The passage from \eqref{eq:mainestimative} to the basic estimate \eqref{estimateconclusion} follows immediately from
Lemma \ref{lem:forms in C_0} and Proposition \ref{prop:Psi_0 terms}.
\end{proof}

%
%
\section{The Proof of Theorem \ref{thm:mainthm, Sobolev}}\label{sec:main theorem, weighted}

Now that we have the tools of Section \ref{sec: basic estimate}, we can prove strong closed range estimates {using many of the
arguments of \cite{HaRa11}. We do, however, use a substantially different elliptic regularization to pay particular attention to 
the regularity of the weighted harmonic forms, the relationship of the harmonic forms with the regularized operators, and an especially
detailed look at the induction base case}. 

\begin{lemma}[Lemma 5.1, \cite{HaRa11}]
	Let $ M $ be a smooth, embedded $ CR $-manifold of hypersurface type that satisfies $Y(q)$ weakly.
	If $ t>0 $ is {suitably large  and the functions $\phi^+,\phi^-$} are as in (\ref{eqn:phi+, phi- defns}), then
\begin{enumerate}[(i)]
	\item $ \opH_{{t}}^q $ is finite dimensional;
	\item There exists $ C $ that does not depend on $ \phi^+ $ and $ \phi^- $ so that for all $ (0,q) $-forms $ u \in \DQ$ 
	satisfying $ u\perp \opH_{{t}}^q $ (with respect to $ \ip{\cdot,\cdot}_{{t}} $) we have
	\begin{equation}\label{eq:c47}
	\normm{u}_{{t}}^2\leq CQ_{b,{t}}(u,u).
	\end{equation}
\end{enumerate}
\end{lemma}

By \cite[Theorem 1.1.2]{Hor65}, $ \pbb:L^2_{0,q}(M {\normm{\cdot}_t})\rightarrow L^2_{0,q+1}(M, {\normm{\cdot}_t})$ and 
$ \dbars_{b,{t}}:L^2_{0,q}(M, {\normm{\cdot}_t})\rightarrow L^2_{0,q-1}(M, {\normm{\cdot}_t}) $ have closed range. Consequently, their adjoints 
$ \dbarb:L^2_{0,q-1}(M, {\normm{\cdot}_t})\rightarrow L^2_{0,q}(M, {\normm{\cdot}_t})$ and 
$ \dbars_{b,{t}}:L^2_{0,q+1}(M, {\normm{\cdot}_t})\rightarrow L^2_{0,q}(M, {\normm{\cdot}_t}) $ 
have closed range as well \cite[Theorem 1.1.1]{Hor65}.

\subsection{{Continuity} of the Green operator $ G_{q,{t}} $}

The {complex Green operator $G_{q,t}$ is the inverse to $\Box_{b,t}$ on $\opH_{q,t}^{\perp}(M)$
(and is defined to be $0$ on $\opH_{q,t}(M)$).}
Recall the following well-known lemma. See, e.g., \cite{FoKo72,Nic06}.
\begin{lemma}\label{lemma:54Nicoara}
	Let $ H $ be a Hilbert space equipped with the inner product $ \pare{\cdot,\cdot} $, corresponding norm $ \|\cdot\|$, 
	and a positive definite Hermitian form  $ Q $ defined on a dense {subset} $ D\subset H $ satisfying
	\begin{equation}\label{eq:c.}
		\norm{\varphi}^2\leq CQ(\varphi,\varphi)
	\end{equation}
	for all $ \varphi\in D $. {Furthermore,} $ D $ and $ Q $ are such that $ D $ is a Hilbert space under the inner product $ Q(\cdot,\cdot) $. 
	Then there exists a unique self-adjoint injective operator  $ F $ with $ \Dom(F) \subset D$ satisfying
	\begin{equation*}
	Q(\varphi,\phi)=(F\varphi,\phi)
	\end{equation*}
	for all $ \varphi\in \Dom(F) $ and $ \phi\in D $. $ F $ is called the Friedrich's representative.
\end{lemma}

In order to use the result above, we prove a density result on ${^\perp \opH_{{t}}^q}(M)$.
\begin{lemma}\label{eq:lemmaclaim1}
 $ \pare{\DQ\cap {^\perp\opH_{{t}}^q}(M),Q_{b,{t}}(\cdot,\cdot)^{1/2}} $ is a Hilbert space (for $ (0,q) $-forms), and $ \DQ\cap {^\perp \opH_{{t}}^q}(M) $ is dense in $ ^\perp \opH_{{t}}^q $.
\end{lemma}
\begin{proof}
Suppose $\{u_\ell\} \subset \DQ\cap {^\perp\opH_{{t}}^q}(M) $ is a Cauchy sequence with respect to the norm $ Q_{b,{t}}(\cdot,\cdot)^{1/2} $. 
Then $ \pbb u_\ell $ and $ \dbars_{b,{t}} u_\ell $ are Cauchy sequences in $ L^2_{0,q+1}(M,{\normm{\cdot}_t}) $ and 
$ L^2_{0,q-1}(M,{\normm{\cdot}_t}) $, respectively, 
so they converge to $ v_1\in L^2_{0,q+1}(M,{\normm{\cdot}_t}) $ and $ v_2\in L^2_{0,q-1}(M,{\normm{\cdot}_t}) $ respectively. 
By \eqref{eq:c47}, this means $ \{u_\ell\} $ is a Cauchy sequence in $ L^2_{0,q}(M,{\normm{\cdot}_t}) $, hence converges to some 
$ u\in L^2_{0,q}(M,{\normm{\cdot}_t}) $. Thus $ u\in \DQ $, $ \pbb u=v_1 $, and $ \dbars_{b,{t}} u=v_2 $ 
since $ \pbb $ and $ \dbars_{b,{t}} $ are closed operators. Since $ 0=\pare{u_\ell,w}_{{t}} $ for all 
$ w\in \opH_{{t}}^q $ and $ \normm{u_\ell-u}_{{t}} \rightarrow 0$,
$ u\in {^\perp\opH_{{t}}^q}(M) $. Thus $ u\in \DQ\cap {^\perp\opH_{{t}}^q }$.

Next, suppose $ u\in {^\perp\opH_{{t}}^q}(M) $ is nonzero and $ u_\ell\in \DQ $ satisfies $ u_\ell\rightarrow u $ on $ L^2_{0,q}(M,{\normm{\cdot}_t}) $. Let 
$ v_\ell=(I-H^q_{{t}} )u_\ell ${, with $ H^q_{{t}} $ the orthogonal projection onto $ \opH_{{t}}^q $  }. The forms $ v_\ell\in {^\perp\opH_{{t}}^q}(M) \cap \DQ $. 
Since $u\neq 0$, it cannot be the case that $ v_\ell =0$ for every {$ \ell $}. 
Since $ \normm{u_\ell}_{{t}}^2=\normm{H^q_{{t}} u_\ell}_{{t}}^2+\normm{v_\ell}_{{t}}^2 $,  and the forms $H^q_{{t}} u_\ell$ and $v_\ell$ are orthogonal,
$ H^q_{{t}} u_\ell $ and $ v_\ell $ both converge in $ L^2_{0,q}(M,{\normm{\cdot}_t}) $. 
Let $ \alpha=\lim_{\ell\to \infty} H^q_{{t}}  u_\ell $, 
$ v=\lim_{\ell\to\infty} v_\ell $, and since that $ H^q_{{t}} u_\ell=u_\ell-v_\ell $, $ \alpha=u-v\in {^\perp\opH_{{t}}^q}(M) $. 
However, $ \alpha\in \opH_{{t}}^q $ since $ \opH_{{t}}^q $ is closed, forcing $ \alpha=0 $. 
Thus, $ \normm{u-v_\ell}_{{t}}\leq \normm{u-u_\ell}_{{t}}+\normm{H^q_{{t}}  u_\ell}_{{t}}\rightarrow 0 $. 
Consequently $ \DQ\cap {^\perp \opH_{{t}}^q}(M) $ is dense in ${^\perp\opH_{{t}}^q}(M) $. 
\end{proof}
We now can establish the existence and $L^2$-continuity of the complex Green operator $G_{q,{t}}$ {using the following well-known result
(we adapt the presentation and argument in \cite[Corollary 5.5]{Nic06}}.
\begin{corollary}\label{cor:G_q exists}
	Let $ M $ be a smooth compact, orientable embedded $ CR $- manifold of hypersurface type that satisfies weak $Y(q)$.
	If $ t >0$ is suitable large, $\phi^+,\phi^-$ are as in (\ref{eqn:phi+, phi- defns}), and $ \alpha\in {^\perp\opH}_{{t}}^q $, then there exists a unique 
$ \varphi_t\in {^\perp\opH}_{{t}}^q\cap \DQ $ such that
	\[
	Q_{b,{t}}(\varphi_t,\phi)=\pare{\alpha,\phi}_{{t}},  \qquad \text {for all } \phi \in \DQ.
	\]
	We define the Green operator $ G_{q,{t}} $ to be the operator that maps $ \alpha $ into $ \varphi_t $. $ G_{q,{t}} $ is a bounded operator, and if additionally $ \alpha $ is closed, then $ u_t=\dbars_{b,{t}} G_{q,{t}}\alpha $ satisfies $ \pbb u_t=\alpha $. We define $ G_{q,{t}} $ to be identically 0 on $ \opH_{{t}}^q $.
\end{corollary}

\subsection{Smoothness of harmonic forms}

Here we will prove that $ \opH_{{t}}^q \subset H^s_{0,q}{(M,\normm{\cdot}_t)}$  for $ t $ sufficiently large.
We adapt the arguments of \cite{KhRa20,HaRa20SCRE}. See also \cite{Nic06,Koh73}. 

Fix $s\geq 1$. For forms $f,g \in H^1_{0,q}{(M,\normm{\cdot}_t)}$, set
\[
Q_{b,{t}}^{\delta,\nu}(f,g) =Q_{b,{t}}(f,g)+\delta Q_{d_b}(f,g)+\nu\pare{f,g}_{{t}}
\]
where $ Q_{d_b}(\cdot,\cdot) $ is the hermitian inner product associated to the Rham exterior derivative $ d_b $, 
i.e., $ Q_{d_b}(u,v)=\pare{d_bu,d_bv}_{{t}}+ \pare{d_{b,{t}}^*u,d_{b,{t}}^*v}_{{t}}${, and $ \delta , \nu\geq0 $ }. Also note that 
$ Q_{b,{t}}^{0,\nu}(f,g) =Q_{b,{t}}(f,g)+\nu\pare{f,g}_{{t}}$ for $ f,g \in \DQ$. Then
\[
\normm{\vp}_{{t}}^2\leq \frac 1\nu Q_{b,{t}}^{\delta,\nu}(\vp,\vp).
\]
for all $\vp \in H^1_{0,q}{(M,\normm{\cdot}_t)}$ if $\delta>0$ and all $\vp\in\DQ$ if $\delta =0$.
By the Lemma \ref{lemma:54Nicoara} there exist self-adjoint operators (for $0 \leq \delta \leq 1$ and $0 < \nu \leq 1$)
$ \Box_{b,{t}}^{\delta,\nu}:\Dom(\Box_{b,{t}}^{\delta,\nu}) \to L^2_{0,q}{(M,\normm{\cdot}_t)} $, 
with inverses 
$ \gqtdn :L^2_{0,q}{(M,\normm{\cdot}_t)}\rightarrow \Dom(\Box_{b,{t}}^{\delta,\nu}) $
satisfying
\begin{equation}
\normm{\gqtdn \vp}^2_{{t}} \leq \frac 1\nu\normm{\vp}_{{t}}^2 \label{eq:cp1} \\
\end{equation}
for all $\vp\in L^2_{0,q}{(M,\normm{\cdot}_t)}$ and all $\delta \in [0,1]$.

{Our goal is to prove
\begin{equation}\label{eq:cp2}
\norm{\gqtzn\vp}_{H^s}\leq K_t\norm{\vp}_{H^s}+C_{t,s}\norm{\gqtzn \vp}_0.
\end{equation}
In fact,}
 \eqref{eq:cp2} is the main tool that we need to prove that $ \opH_{{t}}^q(M)\subset H^s_{0,q}{(M,\normm{\cdot}_t)} $, for $ t $ sufficiently large. 
Given \eqref{eq:cp2}, the argument {for regularity of the harmonic forms} 
follows nearly verbatim from \cite[Proposition 5.2]{Koh73}, from equation (5.20) onwards.
Equation \eqref{eq:cp2} plays the role of \cite[(5.20)]{Koh73}.

{We now prove \eqref{eq:cp2}.}
The operator $\Box_{b,{t}}^{\delta,\nu}$ is elliptic when $\delta>0$ 
which means that  $ \gqtdn: H^s_{0,q}{(M,\normm{\cdot}_t)} \to H^{s+2}_{0,q}{(M,\normm{\cdot}_t)} $.

If $ \varphi \in H^s_{0,q}{(M,\normm{\cdot}_t)} $, then
\[
\norm{\gqtdn\varphi}^2_{H^s}=\norm{\Lambda^s\gqtdn\varphi}_0^2\leq C_t \Norm{\Lambda^s\gqtdn\varphi}_{{t}}^2.
\]

Since $ \gqtdn\varphi\in H^{s+2}_{0,q}{(M,\normm{\cdot}_t)} $, the basic estimate yields
\begin{equation}
\Norm{\Lambda^s\gqtdn\varphi}_{{t}}^2
\leq \dfrac{K}{t}Q_{b,{t}}^{\delta,\nu}(\Lambda^s\gqtdn\varphi,\Lambda^s\gqtdn\varphi)+C_{t,s}\norm{\gqtdn\varphi}_{H^{s-1}} \label{eq:ct01}
\end{equation}

A careful integration by parts shows that
\begin{align*}
&\Norm{\pbb\Lambda^s\gqtdn\vp}_{{t}}^2 \\
&= \big\la \Lambda^s \dbarbsvp \dbarb\gqtdn\vp,\Lambda^s \gqtdn\vp\big\ra 
+ \big\la {\dbarb} \Lambda^s \gqtdn\vp, \big([\Lambda^s,\dbarb] + \Lambda^{-s}\big[[\Lambda^s,\dbarb], \Lambda^s\big]\big) \gqtdn \vp\big\ra\\
&+ \big\la {[\Lambda^s,\dbarb]} \gqtdn\vp, \big([\Lambda^s,\dbarb] + \Lambda^{-s}\big[[\Lambda^s,\dbarb], \Lambda^s\big]\big) \gqtdn \vp\big\ra
+ \big\la[\dbarb,\Lambda^s]\gqtdn\vp,\dbarb\Lambda^s\gqtdn\vp\big\ra.
\end{align*}
We next apply the same 
sequence of integration by parts and commutators to the other terms in $Q_{b,{t}}^{\delta,\nu}(\Lambda^s\gqtdn\vp,\Lambda^s\gqtdn\vp)$.
Using a small constant/large constant argument
and the fact that $\dbarbsvp = \dbarbs +t P_0$ where $P_0$ is a 
(pseudo)differential operator of order $0$, we can absorb terms to obtain
\begin{equation}\label{eqn:Q delta nu arg}	
Q_{b,{t}}^{\delta,\nu}(\Lambda^s\gqtdn\vp,\Lambda^s\gqtdn\vp)
\leq C\normm{\Lambda^s\vp}_{{t}}^2+C_s\normm{\Lambda^s\gqtdn\vp}_{{t}}^2+C_{t,s}\norm{\gqtdn \vp}_{H^{s-1}}
\end{equation}
where $ C $ does not depend  $ t,s,\delta$, or $\nu $, 
and $ C_s $ does not depend on $ t,\delta$, or $\nu $. 
By \eqref{eq:ct01}, for $ t $ sufficiently large
\begin{equation*}
\norm{\gqtdn\vp}_{H^s}^2\leq K_t\norm{\vp}_{H^s}^2+C_{t,s}\norm{\gqtdn\vp}_{H^{s-1}}^2.
\end{equation*}
By induction, we can reduce the $H^{s-1}$-norm to an $L^2$-norm, and by \eqref{eq:cp1}, we observe
\[
\norm{\gqtdn\vp}_{H^s}^2\leq K_t\norm{\vp}_{H^s}^2+C_{t,s,\nu}\norm{\vp}_0^2,
\]
uniformly in $ \delta>0 $. Then there exists a sequence $\{G_{q,{t}}^{\delta_k,\nu}\vp\}_k $ converging weakly to an element 
$ u_\nu $ in $ H^s_{0,q}{(M,\normm{\cdot}_t)} $ when $ \delta_k\rightarrow 0 $, and satisfying both 
\begin{equation}\label{eqn:u_nu H^s_{0,q}(M)ests}
\norm{u_\nu}_{H^s}\leq K_t\norm{\vp}_{H^s}+C_{t,s,\nu}\norm{\vp}_0
\quad\text{and}\quad
\norm{u_\nu}_{H^s}\leq K_t\norm{\vp}_{H^s}+C_{t,s}\norm{u_\nu}_0.
\end{equation}
Since $H^s_{0,q}{(M,\normm{\cdot}_t)}$ embeds compactly in $H^{s'}_{0,q}{(M,\normm{\cdot}_t)}$, it follows that $G_{q,{t}}^{\delta_k,\nu}\vp \to u_\nu$ strongly in
$H^{s'}_{0,q}{(M,\normm{\cdot}_t)}$ for $0 \leq s' < s$. 
Also, observe that the next conclusion is not automatic in the $s=1$ case.
\begin{align}
\Norm{\pbb \gqtdn \vp}_{{t}}^2+\Norm{\dbars_{b,{t}} \gqtdn \vp}_{{t}}^2
&\leq Q_{q,{t}}^{\delta,\nu}(\gqtdn\vp,\gqtdn\vp)\nn \\
&=\pare{\vp,\gqtdn \vp}_{{t}}\leq \normm{\vp}_{{t}}\Norm{\gqtdn\vp}_{{t}}\leq C_\nu\normm{\vp}_{{t}}^2, \label{eq:sqc1}
\end{align}
and, moreover,  $ \pbb G_{q,{t}}^{\delta_k,\nu}\vp $ and $ \dbars_{b,{t}} G_{q,{t}}^{\delta_k,\nu}\vp $ 
are Cauchy sequences in $L^2$. Indeed, assuming $ \delta_k\leq \delta_j $ we have
\begin{align*}
\normm{\pbb\gqt^{\delta_k,\nu}\vp-\pbb\gqt^{\delta_j,\nu}\vp}_{{t}}^2&
+\normm{\dbars_{b,{t}}\gqt^{\delta_k,\nu}\vp-\dbars_{b,{t}}\gqt^{\delta_j,\nu}\vp}_{{t}}^2 \\
&\leq Q_{b,{t}}^{\delta_k,\nu}(\gqt^{\delta_k,\nu}\vp-\gqt^{\delta_j,\nu}\vp,\gqt^{\delta_k,\nu}\vp-\gqt^{\delta_j,\nu}\vp)\\
&=\ip{\vp,\gqt^{\delta_k,\nu}\vp-\gqt^{\delta_j,\nu}\vp}_{{t}}-Q_{q,{t}}^{\delta_k,\nu}(\gqt^{\delta_j,\nu}\vp,\gqt^{\delta_k,\nu}\vp)+Q_{q,{t}}^{\delta_k,\nu}(\gqt^{\delta_j,\nu}\vp,\gqt^{\delta_j,\nu}\vp)\\
&\leq \ip{\vp,\gqt^{\delta_k,\nu}\vp-\gqt^{\delta_j,\nu}\vp}_{{t}}-Q_{q,{t}}^{\delta_k,\nu}(\gqt^{\delta_j,\nu}\vp,\gqt^{\delta_k,\nu}\vp)+Q_{q,{t}}^{\delta_j,\nu}(\gqt^{\delta_j,\nu}\vp,\gqt^{\delta_j,\nu}\vp)\\
&=\ip{\vp,\gqt^{\delta_k,\nu}\vp-\gqt^{\delta_j,\nu}\vp}_{{t}}-\ip{\gqt^{\delta_j,\nu}\vp,\vp}_{{t}}+\ip{\vp,\gqt^{\delta_j,\nu}\vp}_{{t}}\\
&\leq\normm{\vp}_{{t}}\Norm{\gqt^{\delta_k,\nu}\vp-\gqt^{\delta_j,\nu}\vp}_{{t}}
\end{align*}
Since $ \pbb $ and $ \dbars_{b,{t}} $ are closed operators it follows that $ u_\nu \in \DQ $, $ \pbb G_{q,{t}}^{\delta_k,\nu}\vp \to \pbb u_\nu $ 
and $ \dbars_{b,{t}} G_{q,{t}}^{\delta_k,\nu}\vp \rightarrow \dbars_{b,{t}} u_\nu  $ in $L^2$. 
This means $ G_{q,{t}}^{\delta_k,\nu}\vp $ converges strongly to $ u_\nu $ in the 
$ Q_{b,{t}}^{0,\nu}(\cdot,\cdot) ^{1/2} $-norm. Thus,  we will have, for any $ v\in H^2_{0,q}{(M,\normm{\cdot}_t)} $, by \eqref{eq:cp1},
\begin{align*}
\vab{Q_{b,{t}}^{0,\nu}(G_{q,{t}}^{\delta_k,\nu}\vp-\gqtzn \vp,v)}
&= \left|   Q_{b,{t}}^{\delta_k,\nu}(G_{q,{t}}^{\delta_k,\nu}\vp,v)-\delta_k\pare{d_b G_{q,{t}}^{\delta_k,\nu}\vp,d_b v}_{{t}} \right. \\
&\left.-\delta_k\pare{d_{b,{{t}}}^* G_{q,{t}}^{\delta_k,\nu}\vp,d_{b,{{t}}}^* v}_{{t}} - \pare{\vp,v}_{{t}}\right| \\
&= \delta_k\vab{ \pare{G_{q,{t}}^{\delta_k,\nu}\vp,(d_{b,{{t}}}^*d_b+d_bd_{b,{{t}}}^*)v}_{{t}} }
\leq \delta_k C_{\nu,t}\normm{\vp}_{{t}}\norm{v}_2.
\end{align*}
It now follows that $ \gqtzn\vp=u_\nu $ and by (\ref{eqn:u_nu H^s_{0,q}(M)ests}), \eqref{eq:cp2} now follows.
\subsection{Regularity of the Green operator and the canonical solutions.}

In this section we assume $ t $ is sufficiently large {and the} weighted harmonic $(0,q)$-forms, if they exist,
are elements of $H^1_{0,q}(M)\neq \lla{0} $. 
We  use an elliptic regularization argument. {The operator 
$G_{q,{t}}:L^2_{0,q}(M,\normm{\cdot}_t) \to L^2_{0,q}(M,\normm{\cdot}_t) \cap{} ^\perp\opH^q_t(M)$}. 
Consequently, the regularity result {for} $ G_{q,{t}} $ must be on ${^\perp\opH_{{t}}^q}(M)\cap H^s_{0,q}(M)$ for $ s\geq 0 $.
{Continuity on all of $H^s_{0,q}(M)$ then follows because we already established that harmonic forms are elements of  $H^s_{0,q}(M)$}.

The quadratic form $ Q_{q,{t}}^\delta(\cdot,\cdot) := Q_{q,{t}}^{\delta,0}(\cdot,\cdot)$ is an inner product on $H^1_{0,q}(M)$.
By \eqref{eq:c47},
\begin{equation}\label{eq:calpha}
\normm{u}_{{t}}^2\leq CQ_{b,{t}}(u,u)\leq CQ_{b,{t}}^\delta(u,u)
\end{equation}
for all $ u\in H^1_{0,q}(M) \cap{^\perp\opH_{{t}}^q}(M)$. If $ f\in L^2_{0,q}(M) $ then
\[
\vab{\ip{f,g}_{{t}}}\leq \normm{f}_{{t}}\normm{g}_{{t}}\leq \normm{f}_{{t}} C^{1/2}Q_{b,{t}}^\delta(g,g) 
\]
for all $ g\in {^\perp\opH_{{t}}^q}(M) \cap H^1_{0,q}(M)$. This means the mapping $g \mapsto\pare{f,g}_{{t}} $ is a bounded conjugate linear functional on 
${^\perp\opH_{{t}}^q}(M)\cap H^1_{0,q}(M) $. 
By the Riesz Representation Theorem, there exists an element $ G_{q,{t}}^\delta f\in  {^\perp\opH_{{t}}^q}(M)\cap H^1_{0,q}(M) $ 
such that $ \ip{f,g}_{{t}}=Q^\delta_{b,{t}}(G^\delta_{q,{t}}f,g) $ for all $ g\in {^\perp\opH_{{t}}^q}(M)\cap H^1_{0,q}(M)$. 
Moreover, by \eqref{eq:calpha}
	\[
	C^{-1}\normm{G^\delta_{q,{t}}f}_{{t}}^2\leq Q^\delta_{b,{t}}(G^\delta_{q,{t}}f,G^\delta_{q,{t}}f)
	=\ip{f,G^\delta_{q,{t}}f}_{{t}}\leq\Norm{f}_{{t}}\Norm{G^\delta_{q,{t}}f}_{{t}}
	\]
where $ C $ is independent of $ \delta $. Consequently,
\begin{equation}\label{eq:calpha2}
\normm{G^\delta_{q,{t}}f}_{{t}}\leq C\normm{f}_{{t}} 
\end{equation}

Since $ Q^\delta_{b,{t}}(\cdot,\cdot) $ satisfies $ Q^\delta_{b,{t}}(f,f) \geq \delta\normm{\Lambda^1 f}_t^{{2}}$ 
for every $ f\in H^1_{0,q}(M)$, the bilinear form $ Q^\delta_{b,{t}}(\cdot,\cdot) $ is elliptic on $ H^1_{0,q}(M)$.  
This means that $ \varphi \in H^s_{0,q}(M)$ implies $ G^\delta_{q,{t}}\varphi\in H^{s+2}_{0,q}(M) $ ({before, we only knew} that 
$ G^\delta_{q,{t}}\varphi \in {^\perp\opH_{{t}}^q}(M)\cap H^1_{0,q}(M)$).

Let $ \varphi \in H^s_{0,q}(M) $, then
\begin{equation}\label{eq:ct0}
\norm{G^\delta_{q,{t}}\varphi}^2_{H^s}=\norm{\Lambda^sG^\delta_{q,{t}}\varphi}_0^2\leq C_t \normm{\Lambda^sG^\delta_{q,{t}}\varphi}_{{t}}^2.
\end{equation}

We apply the basic estimate to $ G_{q,{t}}^\delta\varphi\in H^{s+2}_{0,q}(M) $ and observe
\begin{equation}\label{eq:ct1}
\normm{\Lambda^sG^\delta_{q,{t}}\varphi}_{{t}}^2\leq \dfrac{K}{t}Q_{b,{t}}(\Lambda^sG^\delta_{q,{t}}\varphi,\Lambda^sG^\delta_{q,{t}}\varphi)+C_{t,s}\norm{G^\delta_{q,{t}}\varphi}_{H^{s-1}}^2.
\end{equation}

Using the argument of \eqref{eqn:Q delta nu arg}, we can establish
\begin{align}
Q_{b,{t}}(\Lambda^s\gqtd\vp,\Lambda^s\gqtd\vp)&\leq Q_{b,{t}}^{\delta}(\Lambda^s\gqtd\vp,\Lambda^s\gqtd\vp)\nonumber \\ 
&\leq C\normm{\Lambda^s\vp}_{{t}}^2+C_s\normm{\Lambda^s\gqtd\vp}_{{t}}^2+C_{t,s}\norm{\gqtd \vp}_{H^{s-1}}^2 \label{eq:cc01} 
\end{align}
where $ C $ is independent of $ t,s,\delta$, and $ \nu $ and $ C_s $ is independent of $ t,\delta$, and $\nu $.

Plugging \eqref{eq:cc01} into \eqref{eq:ct1} and choosing $t$ sufficiently large to absorb terms, we have
\begin{equation}\label{eq:ct3}
\normm{\Lambda^sG^\delta_{q,{t}}\varphi}_{{t}}^2\leq  K_t\norm{\varphi}_{H^s}^2+C_{t,s}\norm{G^\delta_{q,{t}}\varphi}_{H^{s-1}},
\end{equation}
since $ \normm{\Lambda^sG^\delta_{q,{t}}\varphi}_{{t}} <\infty$. Plugging \eqref{eq:ct3} into \eqref{eq:ct0}, it follows that
\[
\norm{G^\delta_{q,{t}}\varphi}^2_{{H^s}}\leq K_t\norm{\varphi}_{H^s}^2+C_{t,s}\norm{G^\delta_{q,{t}}\varphi}_{H^{s-1}}^2.
\]
Using \eqref{eq:calpha2} and induction, we estimate
\begin{equation}\label{eq:ct31}
\norm{G^\delta_{q,{t}}\varphi}_{H^s}^2\leq K_t\norm{\varphi}_{H^s}^2+C_{t,s}\norm{\varphi}_0^2.
\end{equation}

With \eqref{eq:ct31} in hand, we now turn to sending $\delta\to 0$, in a similar manner to 
\cite{HaRa11}. If $ \varphi\in H^s_{0,q}(M)$ then $ \lla{G^\delta_{q,{t}}\varphi : 0<\delta<1} $ is bounded in 
$ H^s_{0,q}(M)$, so there exists $ \delta_k\rightarrow 0 $ 
and $ \tilde{u}\in H^s_{0,q}(M)$ so that $ G^{\delta_k}_{q,t}\varphi\rightarrow \tilde{u} $ weakly in $ H^s_{0,q}(M)$. 
Since the inclusion of $ H^s_{0,q}(M)$ in $ L^2_{0,q}(M) $ is compact, we have $ G^{\delta_k}_{q,t}\varphi\rightarrow \tilde{u} $ 
strongly in $ L^2_{0,q}(M) $ and $ \tilde{u}\in {^\perp\opH_{{t}}^q}(M) $. Also
\begin{equation}\label{eq:ct4}
\norm{\tilde{u}}_{H^s}^2\leq K_t\norm{\varphi}^2_{H^s}+C_{t,s}\norm{\varphi}_0^2.
\end{equation}
Also,
\[
\normm{\pbb G^\delta_{q,{t}}\varphi}^2_{{t}} +\normm{\dbars_{b,{t}} G^\delta_{q,{t}}\varphi}^2_{{t}} \leq Q^\delta_{b,{t}}(G^\delta_{q,{t}}\varphi,G^\delta_{q,{t}}\varphi)=\ip{\varphi,G^\delta_{q,{t}}\varphi}_{{t}}\leq \normm{\varphi}_{{t}}\normm{G^\delta_{q,{t}}\varphi}\leq C_t\normm{\varphi}^2_{{t}} ,
\]
and, as in the previous section, we can prove $ \pbb G^{{\delta_k}}_{q,t}\vp $ and $ \dbars_{b,{t}} G^{{\delta_k}}_{q,t}\vp $ are Cauchy sequences in $L^2_{0,q}(M)$. Since $ \pbb $ and $ \dbars_{b,{t}} $ are closed operators we will have $ \tilde{u} \in \DQ $, $ \pbb \gqtd\vp \rightarrow \pbb \tilde{u} $ and $ \dbars_{b,{t}} \gqtd\vp \rightarrow \dbars_{b,{t}} \tilde{u}  $ in $L^2_{0,q}(M)$, and
\begin{equation}\label{eq:ct5}
\normm{\pbb \tilde{u}}^2_{{t}} +\normm{\dbars_{b,{t}} \tilde{u}}^2_{{t}} \leq C_t\normm{\varphi}_{{t}}^2.
\end{equation}
Consequently if $ v\in H^{s+2}_{{0,q}}{(M)} $, then $ \lim Q^{\delta_k}_{b,t}(G^{\delta_k}_{q,t}\vp,v)=Q_{b,{t}}(\tilde{u},v) $. However, $ Q^{\delta_k}_{b,t}(G^{\delta_k}_{q,t}\varphi,v)=\ip{\varphi,v}_{{t}}= Q_{b,{t}}(G_{q,{t}}\varphi,v)$. So by uniqueness $ G_{q,{t}}\varphi =\tilde{u}$ and \eqref{eq:ct4} 
we  have
\begin{equation}\label{eq:ct41}
\norm{G_{q,{t}}\varphi}_{H^s}^2\leq K_t\normm{\varphi}_{H^s}^2+C_{t,s}\norm{\varphi}_0^2,
\end{equation}
and by \eqref{eq:ct5} 
\begin{equation}
\normm{\pbb G_{q,{t}}\varphi }^2_{{t}} +\normm{\dbars_{b,{t}} G_{q,{t}}\varphi}^2_{{t}} \leq C_t\normm{\varphi}_{{t}}^2.
\end{equation}
These two last equations prove the continuity of $ G_{q,{t}} $ on $ H^s_{0,q}(M)$ 
and as well as $ \pbb G_{q,{t}} $ and $ \dbars_{b,{t}} G_{q,{t}} $ on $L^2_{0,q}(M)$.

The remainder of the proof of Theorem \ref{thm:mainthm, Sobolev} follows from (by now) standard arguments. See, e.g., 
the proof of \cite[Theorem 1.2]{HaRa11}, and Section 6, in particular.

%
%
\section{Proof of the Theorem \ref{thm:mainthm, unweighted}}\label{sec:proof of main theorem}
Since the $L^2(M,{\normm{\cdot}_t})$ and  $L^2(M)$ are equivalent spaces, it is immediate
that $\dbarb :L^2_{0,\tilde q-1}(M)\to L^2_{0,\tilde q} (M)$ has closed range for $\tilde q = q$ or $q+1$. Moreover, by \cite[Theorem 1.1.1]{Hor65},
their adjoints $\dbarbs: L^2_{0,\tilde q}(M)\to L^2_{0,\tilde q-1}(M)$, $\tilde q = q$ or $q+1$ have closed range as well. 
Moreover,  the {dimension of the space of
harmonic $(0,q)$-forms is independent of the weight and is therefore finite (see, e.g., \cite[p.772]{RaSt08} or \cite{Koh73}).}
Standard arguments now establish the rest of Theorem \ref{thm:mainthm, unweighted}.

%
%
\section{Examples}\label{sec:examples}
In this section, we modify the main example of \cite{HaRa15} and show how the flexibility of choosing $\Upsilon$ makes it easier to verify
than the older weak $Y(q)$ condition of \cite{HaRa11}.

Let $M\subset\C^5$ be the boundary of a domain $\Om$ so that on neighborhood $U$ of the origin so that
\[
M \cap U = \{ z = (z_1,\dots, z_5)\in \C^5 : \Imm z_5 = P(z_1,z_2,z_3,z_4)\}.
\]
We set 
\[
\rho(z) = P(z_1,z_2,z_3,z_4) - \Imm z_5
\]
where the polynomial 
\[
P(z_1,z_2,z_3,z_4) = 2x_1|z_2|^2 - x_1y_1^4 + |z_3|^2 + |z_4|^2.
\]
Observe that
\[
\dbar \rho = \Big(|z_2|^2 - \frac 12 y_1^4 - 2ix_1y_1^3\Big)\, d\z_1 + 2x_1z_2\, d\z_2 + \z_3\, d\z_3 + \z_4\, d\z_4 - \frac i2\, d\z_5
\]
and
\[
\p\dbar \rho = -3x_1y_1^2\, dz_1\wedge d\z_1 + z_2\, dz_1\wedge d\z_2 + \z_2\, dz_2\wedge d\z_1 + 2x_1\, d z_2\wedge d\z_2 + dz_3\wedge d\z_3 + dz_4 \wedge d\z_4.
\]
We choose a basis for $T^{1,0}(M\cap U)$ by setting
\[
L_j = \frac{\p}{\p z_j} + 2i \frac{\p P}{\p z_j} \frac{\p}{\p z_5},\quad 1 \leq j \leq 4.
\]
In this basis, we can represent the Levi form by the $4\times 4$ matrix 
\begin{equation}\label{eqn:Levi matrix}
(c_{j\bar k}) = \opL_{\rho_1}(i\LL_k \wedge L_j) = i\p\dbar \rho\Big(i \frac{\p}{\p\z_k}\wedge \frac{\p}{\p z_j}\Big)
= \begin{pmatrix} -3x_1y_1^2 & z_2 & 0 & 0 \\ \z_2 & 2x_1 & 0 & 0 \\ 0 & 0 & 1 & 0 \\ 0 & 0 & 0 & 1 \end{pmatrix} = (\rho_{j\bar k})
\end{equation}
Since $(c_{jk})$ has three positive eigenvalues whenever either $z_2\neq 0$ or both $x\neq 0$ and $y\neq 0$. Hence $Z(2)$ is satisfied on a dense subset of
$M\cap U$.

%
\begin{prop}\label{lem: weak Y(2) near 0}
The CR manifold $M$ satisfies weak $Y(2)$ on $M\cap U$. 
\end{prop}

\begin{proof}The construction of $\Upsilon$ in the proof of \cite[p.1747-1748]{HaRa15} works here as well. 
Moreover, since $\mu_3>0$, it is immediate that we can use the same form $\Upsilon$ for both 
the weak $Z(2)=Z(5-2-1)$ and weak $Z(3)$ cases. 
\end{proof}

Showing that the older weak $Z(2)$ condition fails is quite difficult -- showing that the condition fails in \emph{all} choices of coordinates
amounts to solving a nonlinear problem. Specifically, we know that the signature of the Levi form does not change, but the eigenvalues
certainly can. Computing eigenvalues after coordinate changes or changes of metric is nonlinear and is already quite difficult in the 
$4\times 4$ case. We also point out that none of the weak $Y(q)$ conditions are invariant under the metric as an example from
\cite{HaRa15} shows (no condition that depends on sums of eigenvalues is likely to be invariant under changes of metric).

\bibliographystyle{alpha}
\bibliography{mybib}
\end{document}